\def\draft{n}
\documentclass[11pt]{amsart}
\usepackage[headings]{fullpage}
\usepackage{amssymb,epic,eepic,epsfig,amsbsy,amsmath,amscd,color}
\usepackage[all]{xy}
\usepackage{graphicx}
\usepackage{texdraw}
\usepackage{url}
\usepackage{bbm}
\usepackage{mathalfa}

\def\printname#1{
        \if\draft y
                \smash{\makebox[0pt]{\hspace{-0.5in}
                        \raisebox{8pt}{\tt\tiny #1}}}
        \fi
}

\def\lbl#1{\label{#1}\printname{#1}}

\newtheorem{theorem}{Theorem}[section]
\newtheorem{thm}{Theorem}
\newtheorem{lemma}[theorem]{Lemma}
\newtheorem{corollary}[theorem]{Corollary}
\newtheorem{proposition}[theorem]{Proposition}
\newtheorem{conjecture}{Conjecture}

\newtheorem{definition}[theorem]{Definition}

\theoremstyle{definition}
\newtheorem{remark}[theorem]{Remark}

\def\BC{\mathbb C}

\def\BZ{\mathbb Z}
\def\BR{\mathbb R}

\def\BQ{\mathbb Q}

\def\CP{\mathcal P}
\def\CR{\mathcal R}
\def\CS{\mathcal S}
\def\CT{\mathcal T}

\def\la{\langle}
\def\ra{\rangle}

\DeclareMathOperator{\tr}{\mathrm tr}

\def\al{\alpha}
\def\ve{\varepsilon}
\def\be { \begin{equation} }
\def\ee { \end{equation} }
\def\det{\mathrm{det}}

\begin{document}

\def\bcT{\overline{\mathcal T} }
\def\bcS{\overline{\mathcal{S}} }
\def\bD{{\overline D }}
\def\bTheta{\overline \Theta  }
\def\btheta{\widetilde \theta  }
\def\bcP{\overline{\mathcal P} }
\def\im{\mathrm {Im}}
\def\bs{{\widetilde {\mathfrak s}}}

\def\bbtheta{\overline \theta}
\def\bbs{\overline {\mathfrak s}}
\def\bt{\tilde {\mathfrak t}}

\def\bS{\overline \CS}
\def\bp{{\widetilde {\mathfrak p}}}
\def\bt{{\widetilde {\mathfrak t}}}
\def\wt{\widetilde}
\def\bbp{\overline {\mathfrak p   }}
\def\bbt{\overline {\mathfrak t   }}

\def\CC{\mathcal C}
\newcommand\no[1]{}
\def\div{{\mathrm{div}}}
\def\fin{{\mathrm{fin}}}
\def\hD{{\hat D}}
\def\hP{\hat \CP}
\def\bP{\overline \CP}
\def\hTheta{{\hat \Theta}}
\def\hS{{\hat \CS}}
\def\hT{\hat \CT}
\def\bT{\overline \CT}
\def\cK{\mathcal K}
\def\cQ{\mathcal Q}
\def\Bpi{\BZ[\Pi]}
\def\sC{\mathcal C}
\def\tle{\trianglelefteq}
\def\cS{\mathcal S}
\def\cG{\mathcal G}
\def\fr{{\mathrm{fr}}}
\def\TZ{\mathrm{Tor}_\BZ}
\def\coker{\mathrm{coker}}
\def\ot{\otimes}
\def\Fr{\mathrm{Fr}}
\def\IM{{\mathrm{Im}}}
\def\Mat{{\mathrm{Mat}}}
\def\tX{{\tilde X}}
\def\Det{{\mathrm{det}}}
\def\rk{{\mathrm{rk}}}
\def\vol{\mathrm{vol}}

\def\CC{\mathcal C}
\def\limsupG{\limsup_{\cG \ni \Gamma \to 1}}
\def\limsupGn{\limsup_{\cG^N \ni \Gamma \to 1}}
\newcommand{\rednote}[1]{ {\color{red} [[  #1 ]] }   }
\newcommand{\bluenote}[1]{ {\color{blue} [[  #1 ]] }   }
\def\tY{\tilde Y}
  \def\totr{\overset{\tr}\longrightarrow}
\def\GbP{\Gamma\backslash\Pi}

\title[Torsion growth]{Growth of homology torsion in finite coverings and hyperbolic volume}

\author[Thang  T. Q. L\^e]{Thang  T. Q. L\^e}
\address{School of Mathematics, 686 Cherry Street,
 Georgia Tech, Atlanta, GA 30332, USA}
\email{letu@math.gatech.edu}


\thanks{The author was partially supported in part by National Science Foundation. \\
2010 {\em Mathematics Classification:} Primary 57M27. Secondary 57M25.}

\begin{abstract}
We give an upper bound for the growth of  homology torsions of finite coverings of  irreducible oriented 3-manifolds  in terms of hyperbolic volume.
\end{abstract}

\maketitle

\section{Introduction}
\subsection{Growth of homology torsions in finite coverings}

Suppose $X$ is a connected finite CW-complex with fundamental group $\Pi= \pi_1(X)$.  For a subgroup $\Gamma < \Pi$ of finite index let $X_\Gamma$ be the covering of $X$ corresponding to $\Gamma$, and $t_j(\Gamma)$ be the size of the $\BZ$-torsion part of $H_j(X_\Gamma,\BZ)$.  We want to study the growth of $t_j(\Gamma)$.

A sequence of subgroups $(\Gamma_k)_{k=1}^\infty$ of $\Pi$ is {\em nested} if $\Gamma_{k+1} < \Gamma_k$ and it is {\em exhaustive} if $\cap_k \Gamma_k = \{ e\}$, where $e$ is the unit of $\Pi$. It is known that the fundamental group of any compact 3-manifold  is residually finite \cite{Hempel}, i.e. it has an  exhaustive nested sequence of normal subgroups  of finite index.


 We will prove the following result and some generalizations of it.

\begin{thm} \lbl{thm.1}
Suppose $X$ is an orientable irreducible compact 3-manifold whose boundary $\partial X$ is either an empty set or  a collection of  tori.
Let $(\Gamma_k)_{k=1}^\infty$ be an exhaustive nested sequence of normal subgroups of $\Pi$ of finite index. Then
$$\limsup_{k \to \infty} \frac{\ln t_1(\Gamma_k)}{[\Pi:\Gamma_k]} \le \frac{\vol(X)}{6 \pi}.$$
\end{thm}
Here $\vol(X)$ is defined as follows.  If the fundamental group of $X$ is finite, let $\vol(X)=0$.
Suppose the fundamental group of $X$ is infinite.
By the Geometrization Conjecture of W.~Thurston proved by G.~Perelman (see e.g. \cite{Boileau,BBBMP}), every piece in the  Jaco-Shalen-Johansson decomposition of $X$  is either Seifert fibered or hyperbolic. Define $\vol(X)$ as the sum of the volumes of all hyperbolic pieces.

\subsection{More general limit: trace limit} For any group $G$ with unit $e$ define a trace function $\tr_G: G \to \BZ$ by $\tr_G(g)=1$ if $g=e$ and $\tr_G(g)=0$ if $g \neq e$. This extends to a $\BC$-linear map $\tr_G: \BC[G]\to \BC$. Here $\BC[G]$ is the group ring of $G$ with complex coefficients.

Assume $X$ is a 3-manifold satisfying the assumption of Theorem \ref{thm.1}, with $\Pi=\pi_1(X)$. Let $\cG$ be the set of all finite index subgroups of $\Pi$. For any $\Gamma \in \cG$, $\Pi$ acts on the right on the finite set $\Gamma\backslash\Pi$ of right cosets of $\Gamma$, and the action gives rise to an action of $\Pi$ on $\BC[\GbP]$, the $\BC$-vector space with base $\GbP$. For $g\in \Pi$, define
$$ \tr_{\GbP}(g) = \frac{\tr(g,\BC[\GbP]) }{[\Pi:\Gamma]},$$
where $\tr(g,\BC[\GbP])$ is the trace of the operator $g$ acting on the vector space $\BC[\GbP]$.
Note that when $\Gamma$ is a normal subgroup of $\Pi$, then $G:=\GbP$ is a group, and $\tr_{\GbP}(g)$ coincides with the above definition of $\tr_G(p(g))$, where $p(g)$ is the image of $g$ in the quotient group $G$.

Suppose $(\Gamma_k)_{k=1}^\infty$  is a sequence of subgroups of $\Pi$ of finite index, i.e. $\Gamma_k \in \cG$. We say
$ \Gamma_k \totr 1$
if for any $g\in \Pi$,
$$ \lim_{k\to \infty} \tr_{\Gamma_k\backslash \Pi}(g) = \tr_\Pi(g).$$
For example, if  $(\Gamma_k)_{k=1}^\infty$  is a nested exhaustive sequence of normal subgroups of $\Pi$ of finite index, then $\Gamma_k \totr 1$.

We have the following stronger version of Theorem \ref{thm.1}.
\begin{thm}
\lbl{thm.2}
Suppose $X$ is an orientable irreducible compact 3-manifold whose boundary $\partial X$ is either an empty set or a collection of  tori. Then
\be
\lbl{eq.main}
\limsupG  \frac{\ln t_1(\Gamma)}{[\Pi:\Gamma]} \le \frac{\vol(X)}{6 \pi}.
\ee
\end{thm}

Here $\displaystyle{\limsupG f(\Gamma)}$, for a function $f: \cG\to \BR$,  is defined to be the infimum of the set of all values $L$
such that for any sequence $\Gamma_k \in \cG$ with $\Gamma_k \totr 1$, one has $ \displaystyle{ \limsup_{k\to \infty} f(\Gamma_k) \le L}.$

\begin{conjecture}\lbl{conj.1}
 Suppose $X$ is a  an orientable irreducible compact 3-manifold whose boundary $\partial X$ is either an empty set or a collection of  tori.

 (a)  Let $\cG^N$ be the set of all normal subgroups of $\Pi$ of finite index. One has
$$\limsupGn  \frac{\ln t_1(\Gamma)}{[\Pi:\Gamma]} = \frac{\vol(X)}{6 \pi}.$$

(b) (Weaker version) One has
$$\limsupG  \frac{\ln t_1(\Gamma)}{[\Pi:\Gamma]} = \frac{\vol(X)}{6 \pi}.$$

\end{conjecture}

Similar conjectures were formulated independently by Bergeron and Venkatesh~\cite{BV} and L\"uck~\cite{L4}.

It is clear that the left hand side of \eqref{eq.main} is non-negative. Hence we have the following result.
\begin{corollary} \lbl{r.cor1}
The strong version of Conjecture \ref{conj.1} holds true for 3-manifolds $X$ satisfying the assumption of Theorem \ref{thm.2} with $\vol(X)=0$.
\end{corollary}
For example, if $X$ is the complement in $S^3$ of a tubular neighborhood of a torus link, then the strong conjecture holds for $X$. A 3-manifold $X$ satisfying the assumption of Theorem \ref{thm.2} has $\vol(X)=0$ if and only if $X$ is spherical or a graph manifold.

A statement similar to Corollary \ref{r.cor1}, applicable to a large  class of manifolds,
was proved in \cite{L4}.

\subsection{Branched covering}  We can extend Theorem \ref{thm.2} in the following direction. Suppose $K$ is a knot in $S^3$ and $X=S^3\setminus N(K)$, where $N(K)$ is an open tubular neighborhood of $K$. For a subgroup $\Gamma \le \Pi= \pi_1(X)$ let $\hat X_\Gamma$ be the branched $\Gamma$-covering of $S^3$, branched along $K$. This means, $\hat X_\Gamma$ is obtained from $X_\Gamma$ by attaching solid tori to $\partial X_\Gamma$ in such a way that any lift of a meridian of $K$ bounds a disk in $\hat X_\Gamma$.

\def\tor{\mathrm{Tor}}

\begin{thm}\lbl{thm.2a} Suppose $K$ is a knot in $S^3$. For a subgroup $\Gamma \le \Pi= \pi_1(S^3\setminus K)$ let $\hat X_\Gamma$ be the branched $\Gamma$-covering of $S^3$, branched over $K$. Then
$$ \limsupG  \frac{\ln |\tor_\BZ H_1(\hat X _\Gamma,\BZ)|}{[\Pi:\Gamma]} \le \frac{\vol(X)}{6 \pi} .$$
\end{thm}
Here $|\tor_\BZ H_1(\hat X _\Gamma,\BZ)|$ is the size of the $\BZ$-torsion part of $H_1(\hat X_\Gamma,\BZ)$.

\def\Grhs{{\cG}^{\mathrm{rhs}}}
\def\Ga{\Gamma}
\def\Gkone{\Gamma_k \totr 1}

\subsection{On trace convergence}  In the case when $(\Gamma_k)_{k=1}^\infty$ is a nested sequence of subgroups of $\Pi$ of finite index, the definition of $\Gamma_k \totr 1$ was introduced in \cite{Farber}.
Even when   $(\Gamma_k)_{k=1}^\infty$ is not nested,  $\Gamma_k \totr 1$ if and only if
the representations
$\rho_k: \Pi \to \BC[\GbP]$ with $k=1,2,\dots,$ form an {\em arithmetic approximation of $\Pi$} in the sense of \cite[Definition 9.1]{Farber}.

For a discrete group $\Pi$, there exists a sequence  $(\Gamma_k)_{k=1}^\infty$ of subgroups of finite index such that $\Gamma_k \totr 1$ if and only if $\Pi$ is residually finite.

It turns out that the limit $\Gamma_k \totr 1$, for not necessarily nested sequences, is closely related to known limits in the literature.

{\em The case of normal subgroups.}
When each of $\Gamma_k$ is a normal subgroup of $\Pi$ of finite index, i.e. $\Gamma_k \in \cG^N$, the definition of $\Gkone$ simplifies.
 Suppose $(\Gamma_k)_{k=1}^\infty$ is a nested sequence of normal subgroups of $\Pi$ of finite index. Then
$\Gkone$ if and only if $\cap  \Gamma_k=\{e\}$, i.e. $(\Gamma_k)_{k=1}^\infty$ is exhaustive.

More generally, suppose $(\Gamma_k)_{k=1}^\infty$ is a not necessarily nested sequence of normal subgroups of $\Pi$ of finite index. Then $\Gkone$ if and only if $\cap_{m=1}^\infty \cup_{k=m}^\infty \Gamma_k =\{e\}$. This means, $\Gkone$ if and only if for every non-trivial $g\in \Pi$, $g$ eventually does not belong to $\Gamma_k$.

{\em Relation to sofic approximation.} Suppose $\Gamma_k$ is a sequence of subgroups of finite index of $\Pi$. Then
 $\Gkone$ if and only if $\Sigma = \{\sigma_k\}_{k=1}^\infty$, where $\sigma_k$ is the map from $\Pi$ to the permutation group of $\GbP$, is a {\em sofic approximation} to $\Pi$. For the definition of sofic approximation see \cite{Bowen}.
 Sofic approximation has been important in dynamical system theory. Sofic groups, which are groups having a sofic approximation (not necessarily coming from actions on cosets),  were introduced by Gromov in 1999.

{\em Relation to Benjamini-Schramm convergence.}
Suppose $X$ is a hyperbolic 3-manifold. Raimbault observed that $\Gkone$ if and only if $X_{\Gamma_k}$ Benjamini-Schramm (BS) converge to the hyperbolic space $\mathbb H^3$ in the sense of \cite{ABBGNRS}.
The notion of BS convergence was introduced in \cite{ABBGNRS} for more general sequences of manifolds, not necessarily coverings of a fixed manifold.

\subsection{Growth of Betti numbers} Suppose $X$ is a connected finite CW-complex with fundamental group $\Pi$. For a subgroup $\Gamma \le \Pi$ of finite index let $b_j(\Gamma)$ be the $j$-th Betti number of $X_\Gamma$.  

If $(\Gamma_k)_{k=1}^\infty$ is an exhaustive nested sequence of normal subgroups of $\Pi$, then Kazhdan \cite{Kazhdan,Gromov} showed that
\be
\lbl{eq.Kazh}
\limsup_{k\to \infty} \frac{b_j(X_{\Gamma_k})}{[\Pi:\Gamma_k]} \le b_j^{(2)}(\tilde X),
\ee
where $b_j^{(2)}(\tilde X)$ is the $L^2$-Betti number of the universal covering $\tilde X$. For the definition of $L^2$-Betti number and  $L^2$-invariants in general, see \cite{Luck:book}. L\"uck \cite{Luck:approximation} then showed that one actually has a much stronger result
\be
\lbl{eq.Luck}
\lim_{k\to \infty} \frac{b_j(X_{\Gamma_k})}{[\Pi:\Gamma_k]} = b_j^{(2)}(\tilde X).
\ee

Farber extended L\"uck's result  \eqref{eq.Luck} to the trace convergence, $\Gkone$ (no normal subgroups). The following is a special case of \cite[Theorem 9.2]{Farber}: If $(\Gamma_k)_{k=1}^\infty$ is a sequence of subgroups of $\Pi$ of finite index such that $\Gkone$, then

\be
\lbl{eq.Farber}
\lim_{k\to \infty} \frac{b_j(X_{\Gamma_k})}{[\Pi:\Gamma_k]} = b_j^{(2)}(\tilde X).
\ee

Theorem \ref{thm.2} is an analog of Kazhdan's inequality for the growth of the torsion part of
 the homology, and Conjecture \ref{conj.1} asks for analogs of L\"uck's and Farber's equalities.
In general,  the question about the growth of the torsion part of the homology is
considered much more difficult than  the similar question for the free part of the homology;
 it is closely related
to the difficult question of approximating the Fuglede-Kadison determinant by finite determinants, see
\cite[Question 13.52]{Luck:book} and the discussion there. There has been little progress in this direction.
At the moment of this writing, there is even no known example of a hyperbolic 3-manifold (closed or cusped) and
a sequence $\Gkone$ such that $\displaystyle{
\lim_{k \to \infty} \frac{\ln t_1(\Gamma_k)}{[\Pi:\Gamma_k]} >0}$.

\subsection{Abelian covering}
In the abelian case, results on the growth of the torsion part
of the homology are much more satisfactory, as below.
Proofs of results on the growth of the torsion part of the homology are more complicated than those on
the growth of the
Betti numbers.

Suppose $\tilde X \to X$ is a regular covering with the group of deck transformations equal to $\BZ^n$.
Here $X$ is a finite CW-complex. We choose generators $z_1,\dots,z_n$ of  $\BZ^n$ and identify $\BZ[\BZ^n]\equiv \BZ[z_1^{\pm 1}, \dots, z_n^{\pm 1}]$.

Let $\cG$ in this case be the set of all subgroups $\Gamma \le \BZ^n$ of finite index.
Again $X_\Gamma$ is the corresponding finite covering, and $b_j(\Gamma)$ and $t_j(\Gamma)$, for $\Gamma\in \cG$, are respectively the rank  and the size of the $\BZ$-torsion part of $H_j(X_\Gamma,\BZ)$.
In this abelian case, $\Gkone$ is equivalent to $\lim_{k\to \infty} \la \Gamma_k \ra = \infty$,
where $\la \Gamma \ra$, for $\Gamma \le \BZ^n$, is the smallest among norms of non-zero elements in $\Gamma$.

In this case, we have a positive answer to an analog of Conjecture \ref{conj.1}, as follows. For any $j \ge 0$,
\be
\lbl{eq.Le1}
\limsup_{\BZ^n \ni \Gamma \to 1} \frac{\ln t_j(X_{\Gamma})}{[\BZ^n:\Gamma]} = \mu\left( \Delta(H_j(\tX,\BZ)  ) \right).
\ee
Here $\mu(f)$, for a non-zero $f\in \BZ[z_1^{\pm 1},\dots,z_n^{\pm 1}]$, is the Mahler measure of $f$, and
$\Delta(H_j(\tX,\BZ)  )$ is the first non-trivial Alexander polynomial of $H_j(\tX,\BZ)$, considered as a $\BZ[z_1^{\pm 1},\dots,z_n^{\pm 1}]$-module.
For details and discussions of related results, see \cite{Le:abelian}.
The proof of \eqref{eq.Le1} in \cite{Le:abelian} uses tools from algebraic geometry, commutative algebra, and also results from algebraic dynamics and cannot be generalized to non-commutative cases.

For an arbitrary finite CW-complex with fundamental group $\Pi$ and a fixed  index $j$,
we don't know what would be the upper limit of the left hand side of \eqref{eq.Le1}, with
$\BZ^n$ replaced by $\Pi$. Only when $X$ has some geometric structure,
like 3-dimensional manifolds, do we have results like Theorem \ref{thm.2}.

\subsection{Related results} The growth of torsion parts of homology of coverings or more
generally bundles over a  fixed manifold has attracted a lot of attention lately.
For related results see e.g. \cite{ABBGNRS,BD,BV,BSV,MM,Muller,MeP}.
Probably C. Gordon \cite{Gordon} was the first to ask about the asymptotic growth of
torsion of homology in finite, albeit abelian, coverings. Recently R. Sauer \cite{Sauer}
gave an  upper bound for the growth rate of $t_j(\Gamma)$, for
closed aspherical Riemannian manifolds,
but the bound is not expected to be sharp like in the one in Theorem \ref{thm.2}.

\subsection{Organization} In Section \ref{sec:volume} we recall the definition of
geometric determinant, introduce the notion of metric abelian groups,
 and use them to get an estimate for the torsion parts in exact sequences.
 In Section \ref{sec:L2torsion} we recall the definition of the Fuglede-Kadision
 determinant and prove an upper bound for growth of geometric determinants by Fuglede-Kadison determinant.
  Finally in Section \ref{sec:proof} we give proofs of the main results.

\subsection{Acknowledgements} The author would like to thank N. Bergeron, M. Boileau,
N. Dunfield, S. Friedl, W. L\"uck,  J. Raimbault, A. Reid for enlightening discussions,
J. Meumertzheim for correcting a mistake in Lemma \ref{r.circle}, and anonymous referees for many corrections and suggestions.
The results of the paper were reported in various conferences including the Banff workshop
 ``Low-dimensional Topology and number Theory"~(2007), Fukuoka conference
 ``Low dimensional topology and number theory"~(2009),
 Columbia University conferences ``Geometric Topology" (2009 \& 2013, see  \cite{Le:slides}), and the author would like to thank the organizers of these conferences for inviting him to talk.
The author is partially supported by NSF grant DMS-1105678.

\def\trlim{ \overset \tr \longrightarrow }

\section{Geometric determinant, lattices, and volume in  inner product spaces}
\lbl{sec:volume}
In this section we recall the definition of the geometric determinant, introduce the notion of metric abelian groups, and  establish certain results on upper bounds for torsion parts in exact sequences.
For a finitely-generated  abelian group $A$ let
$\rk(A)$ and $t(A)$ be  respectively the rank of $A$ and the size of the $\BZ$-torsion part of $A$.
For a ring $R$ let
  $\Mat(n \times m, R)$ denote the $R$-module  of all $n\times m$ matrices with entries  in $R$.
\subsection{Geometric determinant}
 For a linear map $f: V_1\to V_2$, where each $V_i$ is a finite-dimensional inner product space the {\em geometric determinant} $\det'(f)$ is defined to be the product of all non-zero
singular values of $f$. Recall that $x\in \BR$ is singular value of $f$ if $x \ge 0$ and $x^2$ is an eigenvalue of $f^* f$. By convention $\det'(f)=1$ if $f$ is the zero map.
We always have $\det'(f) >0$. It is easy to show that $\det'(f)= \det'(f^*)$.

Since the maximal singular value of  $f$ is the norm $||f||$, we have
\be
\lbl{eq.norm}
 {\det}'(f) \le ||f|| ^{\rk(f)} \qquad \text{with convention $0^0=1$.}
 \ee

 If $M\in \Mat(n\times m , \BC)$, we define $\det'(M)= \det'(f)$, where $f: \BC^n \to \BC^m$ is the linear operator defined by $M$, and $\BC^n$ and $\BC^m$ are standard Hermitian spaces.
 \begin{remark}
The geometric meaning of ${\det}' f$ is the following. The map $f$ restricts to a linear isomorphism $f'$ from  $\im(f^*)$ to $\im(f)$, each is an inner product  space.
Then $\det' f= |\det (f')|$, where the ordinary determinant $\det(f')$ is calculated using orthonomal bases of the inner product  spaces.
\end{remark}

\begin{lemma}\lbl{r.cyclic}
Suppose $g$ is a generator of a cyclic group $A$ of order $l$. Define an inner product structure on  $\BC[A]$
such that $A$ is an orthonormal basis. Let $A$ act on $\BC[A]$ by left multiplication.
Then $\det'(1-g)= l$.
\end{lemma}
\begin{proof}
Let $f= 1-g$. Then $f^*= 1-g^{-1}$, and $f^* f= 2 -g -g^{-1}$. Let $\zeta= \exp(2\pi i/l)$. Since $\zeta^k$, with $k=0,1,\dots,l-1$, are  all eigenvalues of $g$, the eigenvalues of $f^*f$ are $2- \zeta^k -\zeta^{-k}= |1-\zeta^k|^2$, with $k=0,1, \dots, l-1$. Excluding the 0 value, we have
$$ \det'(f)=\left| \prod_{j=1}^{l-1} (1-\zeta^l)  \right| =l,$$
where the last identity follows since for any complex number $z$ one has
$$ \prod_{j=1}^{l-1} (z-\zeta^k) = \sum_{k=0}^{l-1} z^k =
\frac{z^l-1}{z-1}.$$
(The above  holds since $\zeta^k$, with $k=1,\dots, l-1$, are roots of $\frac{z^l-1}{z-1}$.)
\end{proof}

\subsection{Volume of lattices} We recall here some well-known facts about volumes of lattices in inner product spaces.

Suppose $V$ is a finite-dimensional inner product space. In this paper, any discrete finite-rank abelian subgroup $\Lambda\subset V$ will be called a {\em lattice} in $V$. Note that we don't require $\Lambda$ to be of maximal rank as in many texts.
For a lattice $\Lambda  \subset V$ with $\BZ$-basis $v_1,\dots,v_l$ define
$$ \vol(\Lambda ) = |\det \left ( (v_i,v_j)_{i,j=1}^l \right)|^{1/2} >0.$$
That is, $\vol(\Lambda)$ is the volume of the parallelepiped spanned by a set of basis vectors.
By convention, the volume of the zero space is 1.

 It is clear that if $\Lambda_1 \subset \Lambda_2$ are lattices in $V$ of the same rank, then
\be
\lbl{eq.u10}
|\Lambda_2/\Lambda_1| = \frac{\vol(\Lambda_1)}{\vol(\Lambda_2)}.
\ee

Let $V$ and  $ W$ be finite-dimensional inner product spaces. Suppose $\Lambda$ is a lattice in $V$ of rank equal the dimension of $V$ and $f: \Lambda\to W$ is
an abelian group homomorphism such that the image $\IM(f)$  is discrete in $W$. The kernel $\ker(f)$ and the image $\IM(f)$ are lattices in respectively $V$ and $W$.
Note that $f$  extends uniquely to a linear map $\tilde f: V \to W$, and we put $\det'(f)= \det'(\tilde f)$. We have

\begin{align} \lbl{eq.detvol}
\vol(\ker f) \,  \vol (\IM(f))  = {\det}' (f) \, \vol(\Lambda).
\end{align}

\def\id{\mathrm{id}}
\subsection{Metric abelian groups} \lbl{sec:metricgroup}
A {\em metric abelian group} is a finitely generated free abelian group $A$ equipped with an inner product on $A\ot_\BZ \BR$.
An {\em integral metric abelian group} is a metric abelian group $A$  such that the inner product  of any two elements $x , y \in A$ is an integer.

Suppose $A$ is a metric abelian group. Then any subgroup
$B\le A$ inherits a metric from $A$. Besides, if $A$ is an integral metric abelian, then the induced metric on $B$ is also integral.

What we will use is the following obvious property of an integral metric group: If $A$ is an integral metric group and $B \le A$, then $\vol(B) \ge 1$.

The following is a basic example.
Suppose $X$ is a finite CW-complex. The cellular $\BZ$-complex $\sC(X)$ of $X$ consists
of free abelian groups $\sC_j(X)$, which are free abelian groups with bases the sets of
cells (of corresponding dimensions) of $X$. We equip $\sC_j(X)$ with a metric such that the mentioned basis
 is an orthonormal basis.  All these metrics are integral.

Suppose $\al: A \to B$ is a group homomorphism between metric abelian groups. Then $\al\ot \id : A \ot_\BZ \BR \to B\ot_\BZ \BR$ is a linear map between finite-dimensional inner product spaces. We define
$$ \det'(\al):= \det' (\al \ot_\BZ \id).$$

\subsection{Torsion estimate} Recall that for a finitely-generated abelian group $A$, $t(A)$ is the size of its $\BZ$-torsion part.
\begin{lemma} \lbl{r.est} Suppose $F_1, F_2$ are integral metric abelian groups.

(a) If  $\al: F_1 \to F_2$ is a group homomorphism, then
\begin{align}
\lbl{eq.ine1}
t(\coker(\al)) \le   \vol(\im(\al)) &\le  {\det}'(\al) \vol(F_1).
\\
\lbl{eq.ine1a}
t(\coker(\al)) \vol(\ker(\al)) & \le {\det}'(\al) \vol(F_1).
\end{align}

(b) If $F_1 \overset \beta \longrightarrow F_2 \to A \to B$ is an exact sequence, where
$A$ and $B$ are finitely-generated abelian groups, then
$$ t(A) \le t(B) \, {\det}' (\beta) \, \vol (F_1).$$
\end{lemma}
\begin{proof} (a) Let $\overline{\im \al} = \left((\im \al)\ot_\BZ \BQ \right) \cap F_2$, where both $(\im\, \al)\ot_\BZ \BQ$ and $ F_2$ are considered as subsets of $F_2 \ot _\BZ \BQ$. Then \eqref{eq.u10} shows that
\be
\lbl{eq.u11}
t(\coker(\al)) = \frac{\vol(\im(\al))}{\vol(\overline{\im \al} )} \le \vol(\im(\al)),
\ee
where the last inequality follows since $\vol(\overline{\im \al} ) \ge 1$ due to the integrality of the metric.
By \eqref{eq.detvol},
$$ \vol(\im(\al))  = \frac{\det'(\al) \vol(F_1)}{ \vol(\ker(\al)) } \le \det'(\al) \vol(F_1),$$
which, together with \eqref{eq.u11}, proves \eqref{eq.ine1}.

(b) First we observe that if
$$ 0\to A_1 \to A_2 \to A_3 \to 0$$
is an exact sequence of finitely generated abelian groups, then
$$ t(A_2) \le t(A_1) t(A_3).$$

From the assumption we have the following exact sequence
$$ 0\to \coker(\beta)  \to A \to B' \to 0,$$
where $B' \le B$. By the above observation,
$$t(A) \le  t(B')\,  t (\coker \beta) \le t(B)\,  t (\coker \beta) \le t(B) \, {\det}' (\beta) \, \vol (F_1),$$
where the last inequality follows from \eqref{eq.ine1a}. This completes the proof of the lemma.
\end{proof}
\subsection{Torsion of a matrix}
 For a matrix $M \in \Mat(m\times n, \BZ)$ let $t(M):= t(\coker M)$, where $M$ is also considered as the homomorphism  $\BZ^m \to \BZ^n$ given by $x \to xM$.

 \def\cC{\mathcal C}
\begin{proposition} \lbl{r.tor5}
 Suppose $M \in \Mat(m\times n, \BZ)$.

(a)  Let $M^*$ be the transpose of $M$, then $t(M)= t(M^*)$.

(b)    One has $t(M) \le \det' (M)$.

(c)   Suppose $(\cC,\partial)$ is a chain complex of free finitely generated $\BZ$-modules. Then $t(H_i(\cC))= t(\coker \partial_{i+1})$.
\end{proposition}
\begin{proof}

(a) Note that  $t(M)$ is equal to the greatest common divisor of all the minors of $M$  of size $r$ where $r=\rk(M)$, see \cite[Section 4.2]{Turaev}. From here we have $t(M)= t(M^*)$.

 Part (b) is a special case of  \eqref{eq.ine1}. For (c), notice that the torsion part of either $H_i(\cC)$ or $\coker \partial _{i+1}$ is equal to $\overline{\im \partial _{i+1}}  / \im \partial_{i+1}$.
\end{proof}
  \def\Tor{\mathrm{Tor}}

\def\cN{\mathcal N}
\section{Fuglede-Kadision determinant} \lbl{sec:L2torsion}
We recall here the definition and establish some properties of the Fuglede-Kadison determinant. We will introduce the Fuglede-Kadison determinant only for a class of operators which we will need in this paper.
For a detailed treatment of the Fuglede-Kadison determinant, the reader should consult the book \cite{Luck:book}.
We prove that the Fuglede-Kadison determinant of a matrix with entries in $\Bpi$ serves as an upper bound for the growth of the geometric determinants of a sequence of finite matrices which approximate the original matrix well enough. This extends a result of L\"uck.

Recall that $\Pi$ is the fundamental group of a finite CW-complex.
\subsection{Fuglede-Kadison determinant of a density function} For simplicity we use the following definition which is more restrictive than the one  in \cite{Luck:approximation}.
\begin{definition}
\lbl{def.det}
(a) A right continuous function
$$F:[0,\infty) \to [0,\infty)$$ is called a {\em density function} if

(i) $F$ is increasing (i.e. $F(\lambda) \le F(\lambda')$ if $\lambda \le \lambda'$), and

(ii) There is a constant $K$ such that $F$ is constant on interval $[K,\infty)$.

\noindent (b) A density function $F$ is said to be in the determinantal class if the integral
$
\int_{0^+}^\infty \ln(\lambda) dF
$
 exists as a real number. If $F$ is in the determinant class,
 define its {\em determinant} by
$$ \det(F):= \exp\left( \int_{0^+}^\infty \ln(\lambda) dF \right).$$
If $F$ is not in the determinantal class, define $\det(F)=0$.
\end{definition}

(In probability theory, a density function is also known as a cumulative distribution function.)

Let $K$ be the number in Condition (ii) of Definition \ref{def.det}. If $F$ is a density function in the determinantal class, then one has (see \cite[Lemma 3.15]{Luck:book})
\be
\lbl{eq.detpart}
 \ln \det(F)= (F(K)-F(0)) \ln(K) -  \int_{0^+}^K \frac{F(\lambda)- F(0)}{\lambda} d \lambda.
 \ee

 For an increasing function $F$ we define
$$ F^+(\lambda)= \lim_{\ve \to 0^+} F(\lambda + \ve).$$
 If $F:[0,\infty) \to [0,\infty)$ is a not necessarily right continuous function satisfying conditions (i)--(ii), then $F^+$, which is $F$ made right continuous, is a density function.

\subsection{Von Neumann algebra of a group and trace function} Let $\ell^2(\Pi)$ be the Hilbert space with orthonormal basis $\Pi$. In other words, $\ell^2(\Pi)$ is the set of all formal sums $\sum_{g\in \Pi} c_g g$, with $c_g \in \BC$ and $\sum_{g\in \Pi} |c_g|^2 < \infty$, with inner product $\la g, g'\ra= \delta_{g,g'}$.
For every positive integer $n$, $(\ell^2(\Pi))^n$ inherits a Hilbert structure, where
$$\la(x_1,\dots,x_n), (y_1,\dots,y_n)\ra= \sum_{j=1}^n \la x_j, y_j \ra.$$
We will consider $(\ell^2(\Pi))^n$ as a left $\Pi$-module by the left multiplication.

By definition, the von Neumann algebra $\cN(\Pi)$ of $\Pi$ is
the $\BC$-algebra of bounded $\Pi$-equivariant operators from $\ell^2(\Pi)$ to $\ell^2(\Pi)$.

For $f\in \cN(\Pi)$ its trace is defined by
$$ \tr_\Pi(f)= \la e, f( e) \ra,$$
where $e$ is the unit of $\Pi$. More generally, suppose $f: (\ell^2(\Pi))^n \to (\ell^2(\Pi))^n$ is a bounded $\Pi$-equivariant operator, define its trace by
$$ \tr(f) = \sum_{j=1}^n \la e_j,  f e_j \ra,$$
where $e_j= 0^{j-1} \times e \times 0^{n-j} \in (\ell^2(\Pi))^n$.

\subsection{Fuglede-Kadision determinants}
Suppose $B$ is an $n\times k$ matrix with entries in $\BC[\Pi]$.
Let $\CR_B: (\ell^2(\Pi))^n  \to (\ell^2(\Pi))^k$ be the bounded $\Bpi$-linear operator defined by $x \to x B$.
Then $(\CR_B)^*\CR_B= \CR_{BB^*}$, where $(\CR_B)^*$ is the adjoint operator, and $B^*$ is obtained from $B$ by the   transpose followed by the conjugation map on $\BC[\Pi]$ given by $ \sum c_i g_i \to  \sum \bar{c_i} g_i^{-1}$, with $\bar{c_i}$ being the complex conjugation of $c_i$.
We define the norm $\|B\|= \| \CR_B\|$.

Let $\{ P(\lambda), \lambda \in [0,\infty)\}$ be the right continuous spectral family of the positive operator $\CR_{BB^*}: (\ell^2(\Pi))^n  \to (\ell^2(\Pi))^n, x \to x B B^*$. {\em The spectral density function of $B$} is defined by
$$F(\lambda)= F_B(\lambda) := \tr_\Pi(P(\lambda)).$$

Then $F:[0,\infty) \to [0,\infty)$ is a density function (Definition \ref{def.det}). We say that $B$ is {\em  in the  determinantal class} if $F$ is in the determinantal class. Define
$$ \det_\Pi(B)= \sqrt {\det (F)}.$$
When $B$ is in the determinantal class $\det_\Pi(B)$ is a positive real number.

\def\nr{\mathrm{nr}}
 \subsection{Relation between Fuglede-Kadison determinant and geometric determinant} \lbl{sec:trivialgroup}

Any matrix $B\in \Mat(n \times k, \BC)$ (with complex entries) can be considered as an element of $\Mat(n \times k, \BC[\Pi])$, where $\Pi$ is the trivial group, since $\BC[\Pi]= \BC$. In this case, $\tr_\Pi$ is the usual trace,  $B$ is always in the determinantal class, and (see \cite[Example 3.12]{Luck:book})
\be
\lbl{eq.trivialgroup11}
 \det_\Pi(B)= \det'(B).
 \ee

 For $\lambda \ge 0$ the spectral density function $F_B(\lambda)$ is the number of eigenvalues of $\CR_{B B^*}$ which are less than or equal to $\lambda$, counted with multiplicity.
The density function  $F_B: [0,\infty) \to [0,\infty)$ is a bounded, right continuous, step function, and $F_B(\lambda)= \nr(B)$, which is the number of rows of $B$,  if $\lambda \ge \|B\|^2$.   Besides, $F_B(0)=\dim \ker(\CR_B)$.

\subsection{Universal bound for norm} \def\cO{\mathcal O}  Suppose $B\in \Mat(n\times k, \BC[\Pi])$. For any  $\Gamma\in \cG$ let $B_\Gamma: \BC[\GbP]^n \to \BC[\GbP]^k$ be the induced $\BC$-homomorphism defined by $x\to xB$. We always equip $\BC[\GbP]$ (and consequently $\BC[\GbP]^n$)  with the Hermitian structure in which $\GbP$ is an orthonormal basis.

One can easily find a universal upper bound for the norm of all operators induced from one acting on $\BZ[\Pi]^n$.
\begin{lemma}
\lbl{r.bound} Suppose $B\in \Mat(n\times k, \BC[\Pi])$. There is a constant $K$ (depending on $B$) such that  $\| B\| <K$ and
$ \|B_\Gamma \| < K
$
for any  $\Gamma\in \cG$. \end{lemma}

\begin{proof}
When $\Gamma$ is a normal subgroup, the statement was proved in \cite{Luck:approximation} with $K=nk \max_{i,j} |B_{ij}|_1$, where for an element $x=\sum c_g g \in \BC[\Pi]$ one sets $|x|_1= \sum |c_g|$. The easy  proof in \cite{Luck:approximation} also works  for our more general case.
\end{proof}

  \def\G{\Gamma}
\subsection{Upper limit of growth of determinants}
Suppose $B\in \Mat(n \times k, \BZ[\Pi])$. In \cite{Luck:approximation}, L\"uck shows that if $(\G_m)_{m\ge 1}$ is an exhaustive nested sequence of normal subgroups of $\Pi$, then
 $$ \limsup_{m\to \infty}  \frac{\ln \det'(B_{\G_m})}{[\Pi:\G_m]}  \le \ln\det_\Pi(B).$$

\def\simz{\overset 0 \sim}
\def\tal{\tilde \al}
\def\hal{\hat \al}
 We will show that a similar result holds if  the sequence $(B_{\G_m})_{m \ge 0}$ is replaced by a sequence of matrices that approximates well enough the matrix $B$.

\begin{definition}  \lbl{def.appro}
 
Let $B$  be a matrix with entries in $\BZ[\Pi]$. A sequence $(B_m,N_m)_{m\ge 1}$ is said to tracely approximate $B$, if each $B_m$ is a matrix with complex entries, each $N_m$ is a positive number, and there exists $K>0$ such that
all the  following conditions (i)--(iii) are satisfied.\\
  (i)  $\|B\|,  \|B_m\| <K$. 
  \\
(ii)
 For every polynomial $p(z)\in \BC[z]$ we have
 $
  \displaystyle{\tr_\Pi(p(B B^*))= \lim_{m\to \infty} \frac{\tr(p(B_m B^*_m))}{N_m}}.
 $ \\
   (iii) If  $F$ is the spectral density of $B$ and $F_m$ is the spectral density of $B_m$, then
  $ F(0)= \displaystyle{ \lim_{m\to \infty} \frac{F_m(0)}{N_m}}.$
\end{definition}

\begin{theorem}\lbl{r.detbound1} Suppose a sequence $(B_m,N_m)_{m\ge 1}$  tracely approximates a matrix of determinantal class $B\in \Mat(n \times k, \BZ[\Pi])$.
 Then
  $$ \limsup_{m\to \infty}  \frac{\ln \det'(B_m)}{N_m}  \le \ln\det_\Pi(B).$$
\end{theorem}
 The proof is a simple modification of L\"uck's proof and will be given in Appendix \ref{sec.a1}. Actually, the definition of trace approximation  is constructed so that L\"uck's proof works.

 As a corollary, we get the following special case.

 \begin{theorem}\lbl{r.detbound}
 Let $B\in \Mat(n \times k, \BZ[\Pi])$ be of   determinantal class and  $(\G_m)_{m \ge 1}$ be a sequence of subgroups of finite index such that $\G_m \totr 1$. Then  $(B_{\G_m}, [\Pi:\G_m])_{m\ge 1}$ tracely approximates $B$.  Consequently,
$$ \limsup_{m\to \infty}  \frac{\ln \det'(B_{\Gamma_m})}{[\Pi:\Gamma]}  \le \ln\det_\Pi(B).$$
\end{theorem}

\begin{proof} Let $K$ be the  number appeared in
 Lemma \ref{r.bound}, then we have (i) of Defininion \ref{def.appro}.

   Let $p(z)\in \BC[z]$. Then $(p(B B^*))_{\G_m} = p (B_m B^*_m)$. Hence we have (ii) by the  definition of  $\G_m \totr 1$.

  The conclusion of (iii), which is a generalization of \cite[Theorem 0.1]{Luck:approximation}, is a special case of \cite[Theorem 9.2]{Farber}. See Equation \eqref{eq.Farber}. In \cite{Farber}, the results are formulated for the case of a chain complex coming from universal covering of a finite CW complex, but the proof there works for the general case when the boundary operators have entries in $\BZ[\Pi]$.
  \end{proof}

   \def\nr{\mathrm{nr}}
 \def\nc{\mathrm{nc}}
 \def\hf{\hat f}
   \def\hg{\hat g}

 \subsection{Perturbation of a sequence of matrices}

 The following statement shows that in many cases, a small perturbation of a sequence
  $(B_m,N_m)_{m\ge 1} $ tracely approximating $B$ gives another sequence which also tracely approximates $B$.  For a matrix $B$ let  $\nc(B)$ and $\nr(B)$ denote respectively the number of columns and the number of rows of $B$.

 \begin{proposition}\lbl{r.detbound3}
  Suppose $(B_m,N_m)_{m\ge 1}$ tracely approximates a matrix $B$ with entries in $\BZ[\Pi]$. For each $m \ge 1$, assume $d_m$ is a positive integer, and 
  $B_m$ is an upper left corner submatrix of 
  a matrix $A_m$ which has integer entries  (i.e. $B_m$ is obtained from an integer matrix $A_m$ by removing several last columns and several last rows). Also assume that
 \begin{itemize}

 \item   there exists $L >0$ such that $\|A_m\| <L$ for all $m$,
 \item  $\nc(A_m) -\nc(B_m)< d_m$ and $\nr(A_m) -\nr(B_m)< d_m$, and
 \be
 \lbl{eq.size}
  \lim_{m\to \infty} \frac{d_m}{N_m}=0.\ee
   \end{itemize}
 Then the sequence $(A_m,N_m)_{m\ge 1}$ also tracely approximates $B$. Consequently,
 \be
 \lbl{eq.upp} \limsup_{m\to \infty}  \frac{\ln \det'(A_m)}{N_m}  \le \ln\det_\Pi(B).\ee
\end{proposition}
\def\hA{\hat A}
\def\hB{\hat B}
 \begin{proof} Replacing $L$ by a bigger number, we can assume that $K<L$, where $K$ is the number appeared in Definition~\ref{def.appro} of $(B_m,N_m)$.
  We will prove (i)-(iii) of Definition  \ref{def.appro} hold for $(A_m, N_m)$, with $K$ replaced by $L$. 

 (i) clearly holds, since $\|B \| < K <L$ and $ \|A_m\| < L$.
 %

 Let us prove (iii). For any matrix $A$ with complex entries,
 \be
 \lbl{eq.0s}
  F_A(0)= \dim \ker(\CR_A) = \nr(A) - \rk(A).
 \ee
 Let $F_m, G_m$ be  respectively the spectral density function of $B_m, A_m$.
Since $B_m$ is a submatrix of $A_m$ and the size difference is $<d_m$, we have
 $$|\nr(B_m) - \nr(A_m)| < d_m, \  0 \le \rk(A_m) - \rk(B_m)  <2 d_m.$$
 From the above inequalities and  \eqref{eq.0s}, we have
 $$ | G_m(0) - F_m(0)|  < 3d_m.$$
 Combining with \eqref{eq.size}, we have
 $\displaystyle{ \lim_{m\to \infty}\frac{G_m(0)}{N_m} = \lim_{m\to \infty}\frac{F_m(0)}{N_m} = F(0)},$ proving (iii).

 Let us prove (ii). The idea is as follows.  Since entries are integers and $\|A_m\| \le L$, on each row or each column of $A_m$ there cannot be more $L^2$ non-zero entries. From here we will show that for all indices $i$ except for a small set,  $(p(B_m B_m^*))_{ii}= (p(A_m A_m^*))_{ii}$. This will show $|\tr(p(B_m B_m^*))- \tr(p(A_m A_m^*))|$ is small compared to $N_m$.

Let go to the details.
\def\Diff{\mathrm{Diff}} Let $\BZ_+= \{ 1,2,3,\dots \}$ and let $\Mat_0(\BZ_+\times \BZ_+, \BC)$ be the set of all $\BZ_+\times \BZ_+$ matrix (with complex entries) with finite support, i.e. all entries are 0 except for a finite number of them. Each matrix  $\al\in  \Mat_0(\BZ_+\times \BZ_+, \BC)$ is a linear endomorphism of the standard Hilbert space $\ell^2$, and hence one can define its norm $\|\al\|$. We define $\tr(\al)= \sum_i \al_{ii}$, which is finite due to the finite support.
 Let $R_i(\al)$ and $C_j(\al)$ be respectively the $i$-th row of $\al$ and the $j$-th column of $\al$.

 Suppose  $\al,\al'\in  \Mat_0(\BZ_+\times \BZ_+, \BC)$. Let
 $$ \Diff(\al,\al')= \{ i\in \BZ_+ \mid {R_i(  \al) } \neq {R_i( {\al'}) } \ \text{or } \  {C_i(  \al) } \neq {C_i( {\al'}) }\}.$$
  We write $\rho(\al,\al') \le (d,L)$ if $\|\al\|, \|\al'\| < L$ and  $|\Diff(\al,\al')|<d$.
  \begin{lemma}
  \lbl{r.lem5}
  Suppose  $\al, \beta,\al',  \beta' \in \Mat_0(\BZ_+\times \BZ_+, \BC)$ and $\rho(\al,\al') \le (d,L)$,  $\rho(\beta,\beta') \le (d,L)$.

  (a) For $c_1, c_2 \in \BC$ one has $\rho(c_1 \al+ c_2 \beta, c_1 \al'+ c_2 \beta') \le (C_1 d, C_1)$, where $C_1=C_1(L, c_1, c_2)$ is a constant depending  only on  $L, c_1, c_2$.

  (b)  If the entries of $\al,\al',\beta,\beta'$ are integers, then $\rho(\al\beta, \al'\beta') \le (C_2 d, C_2)$, where $C_2=C_2(L)$ is a constant depending  only on $L$.
\end{lemma}
   \begin{proof}  (a) Let  $C_1= \max(2, (|c_1|+ |c_2|)L))$.
It clear that $\| c_1 \al+ c_2 \beta \| < (|c_1|+ |c_2|)L)\le C_1$. Since
   $$ \Diff( c_1 \al+ c_2 \beta, c_1 \al'+ c_2 \beta') \subset \Diff(\al,\al') \cup \Diff(\beta,\beta'),$$
   we have   $|\Diff( c_1 \al+ c_2 \beta, c_1 \al'+ c_2 \beta')| < 2d\le C_1 d $. Part (a) is proved.

  (b)  Let $C_2= 4L^4+2$.
  It is clear that $\|\al \beta\|< L^2\le C_2$. Similarly  $\|\al' \beta'\|< C_2$.

   Let $I_1=\Diff(\al,\al'), J_1=\Diff(\beta,\beta')$. By assumption, $|I_1|, |J_1| < d$.
   For a subset $S \subset \BZ_+$ denote $S^c= \BZ_+ \setminus S$.
   If  $i\in (I_1)^c$ and $j\in (J_1)^c$ then $ij$-entry of $\al\beta$ is
    $$ (\al \beta)_{ij}= R_i(\al) \cdot C_j(\beta) =   {R_i(\al')} \cdot    {C_j(\beta')}= (\al'\beta')_{ij}.$$
   This shows $\al\beta$ and $\al'\beta'$ have the same $(I_1)^c \times (J_1)^c$ submatrix, which is denoted by $\gamma$.

   Since $\|\al\beta\| <L^2$, the norm of each row or each column is $<L^2$, which implies on each row or on each column there are at most $L^4$ non-zero entries. It follows that in any collection of $d$ rows (or $d$ columns)  of $\al\beta$, all the entries are 0 except for at most  $d L^4$ of them. From the $(I_1)^c \times (J_1)^c$ submatrix $\gamma$ we can recover the full matrix $\al\beta$  by adding back less than $d$ columns and less than $d$ rows.  In these $d$ columns there are at most $dL^4$ non-zero rows, and let $I_2$ be the set of indexes of those none-zero rows. We have $|I_2| < d L^4$. Now if  $i\not \in (I_1\cup I_2)$ then $R_i(\al\beta)$ is the 0-extension of $R_i(\gamma)$.  Similarly, there is a subset $I'_2\subset \BZ_+$ with $|I'_2| < d L^4$ such that if  $i\not \in (I_1\cup I'_2)$, then $R_i(\al'\beta')$ is the 0-extension of $R_i(\gamma)$. Hence if $i\not \in (I_1\cup I_2 \cup I'_2)$, then $R_i(\al\beta)= R_i(\al'\beta')$.

   Similarly, there are subsets $J_2, J_2'\subset \BZ_+$ with $|J_2|, |J'_2| < d L^4$ such that if $j\not \in (J_1\cup J_2 \cup J'_2)$, then $C_j(\al\beta)= C_j(\al'\beta')$.

   Let $I= I_1\cup I_2 \cup I'_2 \cup J_1\cup J_2 \cup J'_2$. Then $|I| \le 2d+ 4 dL^4\le d C_2 $.
     We have $R_i(\al\beta)= R_i(\al'\beta')$ and $C_i(\al\beta)= C_i(\al'\beta')$ if $i\not \in I$. This implies $\Diff(\al\beta, \al'\beta') \subset I$, and
   $$|\Diff(\al\beta, \al'\beta') |\le |I| < d C_2 .$$
  This completes the proof of the lemma. \end{proof}
  Let us continue with the proof of (ii). For a finite matrix $\al$ let $\hal\in \Mat_0(\BZ_+\times \BZ_+, \BC)$ be its 0-extension, i.e. $(\hal)_{ij} =\al_{ij}$ if $\al_{ij}$ exists, otherwise $(\hal)_{ij} =0$.  The map $\al\to \hal$ is linear and preserves the trace, the norm, and the matrix product.
  Since $A_m$ is obtained from $B_m$ by adding less than $d_m$ rows and then less than $d_m$ columns, and on each row or each column there are no more than $L^2$ non-zero elements, we have
  $$|\Diff(\widehat{B_m}, \widehat{A_m}) | < d_m + 2L^2d_m = d_m(2L^2+1).$$
  It follows that
  $$\rho(\widehat{B_m}, \widehat{A_m}) \le (d_m(2L^2+1), L), \quad \rho(\widehat{B_m^*}, \widehat{A_m^*}) \le (d_m(2L^2+1), L).$$
    Applying part (b) of  Lemma \ref{r.lem5} repeatedly to products of $\widehat{B_m}, \widehat{B_m^*}, \widehat{A_m}, \widehat{A_m^*}$, then applying part (a)  of Lemma \ref{r.lem5}, we see that there is a constant  $C=C(L,p)$ depending only on $L$ and the polynomial $p$, such that
  $$ \rho(\widehat{p(B_m B_m^*)}, \widehat{p(A_m A_m^*)}  ) \le (C d_m, C).$$
 The absolute value of each entry of either $\widehat{p(B_m B_m^*)}$ or $\widehat{p(A_m A_m^*)}$ is less than $C$ since $C$ is an upper bound for the norm.

  From the definition, if $i\not \in \Diff(\widehat{p(B_m B_m^*)}, \widehat{p(A_m A_m^*)  })$, then $(p(B_m B_m^*))_{ii}= (p(A_m A_m^*))_{ii}$.  Since  $| \Diff(\widehat{p(B_m B_m^*)}, \widehat{p(A_m A_m^*)  }) |< C d_m$  and the absolute value of each entry is less than $C$, we have
  $$ |\tr (\widehat{p(B_m B_m^*)} ) - \tr (\widehat{p(A_m A_m^*)}) | <2 C ( C d_m) = 2 C^2 d_m.$$
  Since $\lim_{m\to \infty } d_m/N_m=0$ and $\tr(\al) = \tr (\hal)$, we conclude that
   $$   \lim_{m\to \infty} \frac{ \tr (p(B_m B_m^*)) }{N_m} = \lim_{m\to \infty} \frac{ \tr (p(A_m A_m^*)) }{N_m} = \tr_\Pi (p(B B^*)),$$
    which completes the proof of (ii). \end{proof}

\begin{remark} 
\lbl{rem.upp}
Since ${\det}' (A_m)$ does not change if one permutes the rows or the columns of $A_m$, Inequality \eqref{eq.upp} still holds if in  Proposition \ref{r.detbound3} we replace the assumption ``$B_m$ is an upper left corner submatrix of $A_m$" by the weaker assumption ``$B_m$ is a submatrix of $A_m$". (Actually, one can also prove that Proposition \ref{r.detbound3} holds under this weaker assumption.)
\end{remark}

\def\cN{{\mathcal N}}

\subsection{Hyperbolic volume and Fuglede-Kadison determinant} \lbl{sec:goodp}
Suppose $X$ is an irreducible orientable compact 3-manifold with infinite fundamental group and with boundary $\partial X$  either empty or a collection of tori. We define now {\em  good presentations} of $\Pi=\pi_1(X)$ and their {\em reduced Jacobians}.

First assume that $\partial X \neq \emptyset$.
Then $X$ is  homotopic to a 2-dimensional finite  CW-complex $Y$ which has one zero-cell. Suppose there are $n$ two-cells. Since the Euler characteristic is 0,
there must be $n+1$ one-cells, denoted by $a_1,\dots, a_{n+1}$. Then $\Pi:= \Pi_1(X) =\Pi_1(Y)$ has a presentation 
\be
\lbl{eq.good1}
 \Pi =\la a_1,\dots, a_n, a_{n+1} \mid r_1, \dots, r_n \ra,
 \ee
 where $r_j$ is the boundary of the $j$-th two-cell. We call such a presentation  a {\em good presentation for $\Pi$} if $a_{n+1}$  is an element of infinite order in $\Pi$.
 For a good presentation \eqref{eq.good1}, define its {\em reduced Jacobian} to be the square matrix
$$ J = \left( \frac{\partial r_i}{\partial a_j}  \right)_{i,j=1}^n \in \Mat(n \times n, \Bpi),$$
where $\frac{\partial r_i}{\partial a_j}$ is the Fox derivative.

Now assume that   $X$ is a closed oriented 3-manifold. Assume $X=H\cup H'$ is a Heegaard splitting of $X$, where each of $H$ and $H'$ is a handlebody of genus $g$. There is a graph $G\subset H$ with 1 vertex and $g$ loop-edges $a_1,\dots, a_g$ such that $H$ is a regular neighborhood of $G$. Similarly there is a graph  $G'\subset H'$ with 1 vertex and $g$ loop-edges $a'_1,\dots, a'_g$ such that $H'$ is a regular neighborhood of $G'$. There is a collection $\{ D'_1,\dots, D'_g\}$ of properly embedded disks in $H'$ which cuts $H'$ into balls  such that $D'_i$ meets $a'_i$ transversally at 1 point and does not meet $a'_j$ for $j\neq i$. We assume that $a_g$ and $a'_g$ are non-trivial in $\pi_1(X)$.
A presentation of $\Pi=\pi_1(X)$ can be given by
\be \lbl{eq.good2}
 \Pi =\la a_1,\dots, a_g \mid r_1, \dots, r_g \ra,
 \ee
 where $r_i$ is given by the boundary of $D'_i$, see e.g. \cite{Hempelbook}. We call such a presentation a {\em good presentation} of $\pi_1(X)$, and 
 define its {\em reduced Jacobian}  to be the square matrix
$$ J = \left( \frac{\partial r_i}{\partial a_j}  \right)_{i,j=1}^{g-1} \in \Mat((g-1) \times (g-1), \Bpi).$$

 We quote here an important result which follows from a result of L\"uck and Schick \cite{LS} (appeared as Theorem 4.3 in \cite{Luck:book}) and L\"uck \cite[Theorem 4.9]{Luck:book}.

\begin{theorem}\lbl{thm.3}
Suppose $X$ is an irreducible orientable compact 3-manifold with infinite fundamental group and with boundary either empty or a collection of tori. Let $J$ be the reduced Jacobian  of a good presentation of $\Pi=\pi_1(X)$.
Then
$$ \ln \det_\Pi(J)=  \vol(X)/6\pi.$$
\end{theorem}
\begin{proof} In the proof of Theorem 4.9 of \cite{Luck:book} appeared in \cite{Luck:3manifold} it was explicitly shown that $\ln\det_\Pi(J)= -\rho^{(2)}(\tilde X) $, where $\rho^{(2)}(\tilde X)$ is the (additive) $L^2$-torsion of the universal covering of $X$.

By \cite[Theorem 4.3]{Luck:book}, $\rho^{(2)}(\tilde X)=  -\vol(X)/6 \pi$. Hence $\ln\det_\Pi(J) = \vol(X)/6\pi$.
\end{proof}

\begin{remark}
In the formulation of \cite[Theorem 4.3]{Luck:book} there is an assumption that $X$ satisfies the conclusion of Thurston Geometrization Conjecture, which is redundant now due Perelman's celebrated result. Besides, there is a requirement that the boundary of $X$ be incompressible. But if a torus component of $X$ is compressible, then $X$ must be a solid torus, for which all the results are trivial.
\end{remark}

\section{Proofs of main results} \lbl{sec:proof}

\subsection{Growth of functions} Recall that $\cG$ is the set of all subgroups of $\Pi$ of finite index. Suppose $f,g: \cG \to \BR_{>0}$ are functions on $\cG$ with positive values. We say that $f$ has {\em negligible growth} if
$$ \limsupG \left (f(\Gamma) \right)^{1/[\Pi:\Gamma]} \le 1.$$
We will write
$ f \tle g$
if $f/g$ has negligible growth.

\subsection{Chain  complexes of coverings} \lbl{sec:chainC}
Suppose $X$ is a connected finite CW-complex with fundamental group $\Pi$, and $\tX$ is its universal covering. Then $\tX$ inherits a CW-complex structure from $X$, where the cells of $\tX$ are lifts of cells of $X$. The  action of $\Pi$ on $\tX$ preserves the CW-structure and commutes with the boundary operators. Let $\sC(\tX)$ be the  chain $\BZ$-complex  of the CW-structure of $\tX$. Then  $\sC_j(\tX)$ is the free $\BZ$-module with basis the set of all $j$-cells of $\tX$. We will identify $\sC_j(\tX)$ with  $\Bpi^{n_j}$, a free $\Bpi$-module, as follows.

We assume that

(i) $X$  has only one 0-cell $e^0$, and

(ii) for any $j$-cell $e$  of $X$ with $j\ge 1$,  one has $e^0= \chi_e((1,0,\dots,0))$, where
$\chi_e: D^j \to X$ is the  characteristic map of $e$. Here
$ D^j=\{ x \in \BR^j, \|x \| \le 1\}$
is the standard unit $j$-disk.

Choose a lift $\tilde e^0$ of $e^0$. Condition (ii) shows that  for every $j$-cell $e$
there is a unique lift $\tilde e$ defined by using the lift of the characteristic map which
sends $(1,0,\dots,0)$ to $\tilde e^0$.
Let $e_1^j,\dots,e^j_{n_j}$ be an ordered set of all $j$-cells of $X$, then
$$ \sC_j(\tX)= \bigoplus_{l=1}^{n_j} \Bpi \cdot \tilde e^j_{l},$$
and we use the above equality to identify $\sC_j(\tX)$ with $\Bpi^{n_j}$.

\no{Thus, the identification $\sC_j(\tX)\equiv\Bpi^{n_j}$, for a CW-complex $X$ satisfying (i) and (ii) depends on the choice of a lift of $e^0$ and an ordering on the set of $j$-cells, for every $j\ge 1$.
}

We will write an element $x\in \Bpi^n$ as a row vector $x=(x_1,\dots,x_n)$ where each $x_j \in \Bpi$. Note that $\Bpi^n$ can be considered as a left $\Bpi$-module or a right $\Bpi$-module. We will consider $\Bpi^n$ as a left $\Bpi$-module unless otherwise stated.
If $B\in \Mat(n\times m, \Bpi)$ is an $n\times m$ matrix with entries in $\Bpi$, then the right multiplication by $B$ defines a $\Bpi$-linear map from $\Bpi^n$ to $\Bpi^m$, and every $\Bpi$-linear map $\Bpi^n \to \Bpi^m$ arises in this way.

The boundary operator $\partial_j: \sC_j(\tX) \to \sC_{j-1}(\tX)$ is given by a $n_j \times n_{j-1}$ matrix with entries in $\Bpi$; by abusing notation we also use $\partial_j$ to denote this matrix.
The chain complex $\sC(\tX)$  has the form
$$
\sC(\tX)= \left(  \dots \to \Bpi^{n_{j+1}} \overset{\partial_{j+1}} \longrightarrow \Bpi^{n_{j}} \overset{\partial_{j}} \longrightarrow \Bpi^{n_{j-1}} \to \dots \Bpi^{n_1}  \overset{\partial_{1}} \longrightarrow \Bpi^{n_0}  \overset{\partial_{0}} \longrightarrow 0\right).
$$

The 2-skeleton of $X$ gives a presentation of the fundamental group
\be
\pi_1(X)= \la a_1,\dots, a_n \mid  r_1, \dots , r_m\ra,
\ee
where $a_i$ is the represented by the 1-cell $e^1_i$, and $r_j$ is the boundary of the 2-cell $e^2_j$, written as a product of $a_i$'s.  In this case, $\partial_1$ is the $n\times 1$ matrix whose $i$-entry is $1-a_i$, and $\partial_2$ is the $m \times n$ matrix whose $ij$-entry is $\frac{\partial r_i}{\partial a_j}$, see \cite[Claim 16.6]{Turaev}.

Suppose $\Gamma\le \Pi$ is a subgroup and $X_\Gamma$ is the corresponding covering. Then $X_\Gamma$ inherits a CW-structure from $X$, and its chain $\BZ$-complex is exactly $\BZ[\GbP] \ot_{\Bpi} \sC(\tX)$. Here we consider $\BZ[\GbP]$ as a right $\Bpi$-module.

In general, if $\sC$ is a chain complex over $\Bpi$ of left $\Bpi$-modules and $\Gamma\le \Pi$, then we denote by $\sC_\Gamma$ the chain $\BZ$-complex $\BZ[\GbP] \ot_{\Bpi} \sC$.

\subsection{Circle complex} Fix a non-trivial element $a$ of a residually finite group $\Pi$. Let $\cS$ be  the following chain $\BZ[\Pi]$-complex
$$ 0 \to \BZ[\Pi] \overset{\partial_1}\longrightarrow \BZ[\Pi] \overset{\partial_0}\longrightarrow 0,$$
where $\partial_1= 1-a$.
For every $\Gamma \in \cG$, one has the  $\BZ$-complex  $\cS_\Gamma = \BZ[\GbP] \ot_{\BZ[\Pi]} \cS$.
We provide $\BZ[\GbP]$ with a metric such that $\GbP$ is an orthogonal lattice.
Recall that $b_j(\cS_\Gamma)$ is the rank of $H_j(\cS_\Gamma)$.

\begin{lemma} \lbl{r.circle} Suppose $a$ has infinite order as an element of $\Pi$.

(a) One has
\be \lbl{eq.2016b}
\limsupG \frac{b_0(\cS_\Gamma)}{[\Pi:\Gamma]} = \limsupG \frac{b_1(\cS_\Gamma)}{[\Pi:\Gamma]}=0.
\ee

(b)
The function $\Gamma \to {\det}'(1-a_\Gamma)$ is negligible.
\end{lemma}
\begin{proof} (a) Since the order of $a$ is infinite, the group $\la a\ra$ is an infinite cyclic subgroup of $\Pi$. By decomposing $\Pi$ as the disjoint union of cosets of $\la a \ra$,  it is easy to see that the $L^2$-Betti numbers of $\cS$ are 0. Alternatively, the $L^2$-Betti numbers of the circle are all 0, and hence the $L^2$-Betti numbers of $\cS$ are 0 by the induction theorem, see \cite[Theorem 1.35(10)]{Luck:book}.
Hence Lemma 4.1(a) is a special case of  the main result of Farber \cite{Farber}.

Actually, (a) is much simpler than the full result of \cite{Farber} due to the simple nature of $\cS$,
and here is a direct proof (supplied by J. Meumertzheim).
Let  $a_\Gamma$ denote the action of $a$ on $\BZ[\GbP]$. The chain complex $\cS_\Gamma$ has the form
$$ 0 \to \BZ[\GbP] \overset{1-a_\Gamma}\longrightarrow \BZ[\GbP] \longrightarrow 0.$$
From the Euler characteristic consideration, one has $b_1(\cS_\Gamma)= b_0(\cS_\Gamma)$.

 The element $a_\Gamma$ acts on $\BC[\GbP]$ by permuting the basis
 $\GbP$. Suppose as a permutation, $a_\Gamma$ has $d_n(a_\Gamma)$ cycles of length $n$, with total
 $d= d(\Gamma):= \sum_n d_n(a_\Gamma)$ cycles. From $H_1(\sC_\Gamma) = \ker(1-a_\Gamma)$ one can easily see that $b_1(\cS_\Gamma)=d$. We have
 \be
 \lbl{eq.2016}
  [\Pi: \Gamma]= \sum_n n d_n(a_\Gamma).
  \ee
 The trace of $a_\Gamma$ is equal to the number of cycles of length 1, i.e. 
 \be
 \lbl{eq.2016d}
 [\Pi:\Gamma] \tr _{\GbP} (a) = \tr( a_\Gamma, \BC[\GbP])= d_1(a_\Gamma).
 \ee
 Using \eqref{eq.2016d} and counting the number of cycles of length 1 of $(a^m)_\Gamma$, one has
 $$ [\Pi:\Gamma] \tr _{\GbP} (a^m)= d_1( (a^m)_\Gamma  )= \sum_{n|m, n \ge 1} n d_n(a_\Gamma).$$
 In particular,
 \be
 \lbl{eq.2016a}
 d_n(a_\Gamma) \le \frac1n  \, [\Pi:\Gamma] \tr _{\GbP} (a^n).
 \ee
For all positive integers $l$ we have
\begin{align*}
\frac{b_1(\cS_\G)}{[\Pi:\Gamma]}= \frac{\sum_n  d_n(a_\Gamma)}{[\Pi:\Gamma]}&= \frac{\sum_{n=1}^l  d_n(a_\Gamma)}{[\Pi:\Gamma]} + \frac{\sum_{n \ge l+1}  d_n(a_\Gamma)}{[\Pi:\Gamma]} \\
&\le \sum_{n=1}^l  \frac{ \tr _{\GbP} (a^n)}{n } + \frac{\sum_{n \ge l+1}  d_n(a_\Gamma)}{\sum_n n d_n(a_\Gamma)} \qquad \text{by \eqref{eq.2016} and \eqref{eq.2016a}}\\
&\le \sum_{n=1}^l  \frac{ \tr _{\GbP} (a^n)}{n } + \frac{1}{l+1}.
\end{align*}
Taking the limit, using the fact that $\displaystyle{\limsupG \frac{ \tr _{\GbP} (a^n)}{ [\Pi:\Gamma]}= \tr_\Gamma(a^n)=0}$ (since $a^n \neq 1$), we have
$$ \limsupG \frac{b_1(\cS_\G)}{[\Pi:\Gamma]} \le \frac{1}{l+1}.$$
Since this is true for all $l$, we have \eqref{eq.2016b},
 which proves (a).

 (b)
 Let $N=N(\Gamma):=[\Pi:\Gamma]$. Recall that $b_1(\cS_\Gamma)=d(\Gamma)$, the number of cycles of $a_\Gamma$.

 \noindent {\em Claim 1.} The function $(N(\G)/d(\G))^{d(\G)}$ on $\cG$ is negligible.\\
 Proof of Claim 1. By  part (a),  $d/N \to 0$ as $\Gamma \trlim 1$.
  It follows that $(N/d)^{(d/N)} \to 1$ as $\Gamma \trlim 1$. This proves  Claim 1.

Let  $a_\Gamma$ have $d$ cycles of length $l_1,\dots,l_d$.
  Lemma \ref{r.cyclic} shows that
 $$ \det'(1-a_\Gamma)= \prod_{j=1}^d l_j.$$
 Since $\sum_{j=1}^d l_j=N$, the arithmetic-geometric mean inequality implies
  $$ \det'(1-a_\Gamma)= \prod_{j=1}^d l_j \le  (N/d)^{d},
 $$
  which, in light of Claim 1, proves that $\det'(1-a_\Gamma)$ is negligible.
\no{

  Let us continue with the proof of the lemma.
 Let $W\subset \BC[\GbP]$ be the subspace span by $\IM(1-a_\Gamma)$. Then  one can identify $\Fr(H_0(\cS_\Gamma))= W^\perp \cap \BZ[\GbP] $, and the metric on $\Fr(H_0(\cS_\Gamma)$ is exactly the metric of $\BC[\GbP]$ restricted on $W^\perp$. We have
 $$ \vol(H_0(\cS_\Gamma))= \vol (W^\perp \cap \BZ[\GbP] ) = \vol(W \cap \BZ[\GbP]),$$
 where the second identity follows from \eqref{eq.detortho}. Since $\IM(1-a_\Gamma)$ is a sublattice of $W \cap \BZ[\GbP]$, we have $\vol(W \cap \BZ[\GbP]) \le \vol(\IM(1-a_\Gamma))$. Hence,
 \be
 \lbl{eq.17}
 \vol(H_0(\cS_\Gamma))  \le \vol(\IM(1-a_\Gamma)).
 \ee

  Recall that  $H_1(\cS_\Gamma)= \ker(1-a_\Gamma)$. Using \eqref{eq.17}, we have
  \begin{align*}
   \vol(H_1(\cS_\Gamma)) \vol(H_0(\cS_\Gamma)) &\le \vol(\ker(1-a_\Gamma)) \, \vol(\IM(1-a_\Gamma))
  \\
  &= \det'(1-a_\Gamma) \quad \text{by \eqref{eq.detvol}}.
   \end{align*}
   Since each of $ \vol(H_1(\cS_\Gamma))$ and $\vol(H_0(\cS_\Gamma))$ is $\ge 1$, the above shows that
$$
1\le \vol(H_1(\cS_\Gamma)) \le (N/d)^d, \quad 1\le \vol(H_0(\cS_\Gamma)) \le (N/d)^d.
$$
From Claim 2 we can conclude that both $\vol(H_0(\cS_\Gamma)) $ and $\vol(H_1(\cS_\Gamma)) $ are negligible.
}
\end{proof}

\def\ta{\tilde a}

\begin{remark} Concerning the assumption on the order of $a$, we have the following.
Suppose $X$ is an irreducible 3-manifold whose fundamental group $\Pi$ is infinite. Then $\Pi$ does not have
any non-trivial torsion element (see e.g. \cite[(C3) Section 3.2]{AFW}), i.e. if $a\in \Pi$ is not the unit,
then the order of $a$ is infinite.
 \end{remark}

\subsection{Proof of Theorem \ref{thm.2} for the case $\partial X\neq \emptyset$} \lbl{sec:none}
\begin{proof}
Suppose $X$ satisfies the assumption of Theorem \ref{thm.2} and $\partial X\neq \emptyset$.
 As in Section \ref{sec:goodp}, choose a finite 2-dimensional CW-complex $Y$ homotopic to $X$ which gives a good presentation for $\Pi$:
$$ \Pi =\la a_1,\dots, a_n, a_{n+1} \mid r_1, \dots, r_n \ra.$$

 Let $\sC= \sC(\tY)$ be the chain $\BZ$-complex of the CW-structure of the
 universal covering $\tY$ of $Y$. As in  Section \ref{sec:chainC},
$\sC$ has the form
$$
\sC= \left( 0\to \BZ[\Pi]^n \overset{\partial_2}{\longrightarrow} \BZ[\Pi]^{n+1}  \overset{\partial_1}{\longrightarrow}  \BZ[\Pi] \to 0\right),
$$
where $\partial_2 \in \Mat(n \times (n+1), \Bpi)$ is the $n\times n+1$ matrix
 with entries $(\partial_2)_{ij} = \frac{\partial r_i}{ \partial a_j}$ and
 $\partial_1 \in \Mat(n+1 \times 1, \Bpi ) $  is a column vector with entries  $(\partial_1)_j =1-a_j$.
The reduced Jacobian  $J$, defined in Section \ref{sec:goodp}, is  the matrix obtained from $\partial_2$ be removing the last column, which will be denoted  by $c$.

Then $\sC=\sC(\tY)$ is the middle row of  the following commutative diagram

$$\begin{CD}     0     @ >>>          0   @ >  >>    \Bpi @ > 1-a_{n+1} >>    \Bpi @ > \partial_0>> 0 \\
  @ VVV @VV V  @VV \iota V @VV\id V @VVV\\
           0   @>>>      \Bpi^{n}  @  > \partial_2 >> \Bpi^{n+1} @ > \partial_1 >>    \Bpi   @ > \partial_0>> 0 \\
            @ VVV @VV \id V  @VV   p              V @VV  V @VVV\\
           0 @ >>>   \Bpi^n         @ > J >>     \Bpi^n  @ >>> 0 @ >\partial_0 > > 0
\end{CD}
$$
where $\iota: \Bpi\to \Bpi^{n+1}$ is the embedding into the last component, and $p: \Bpi^{n+1}\to \Bpi^n$ is the projection onto the first $n$ components.
Each row is a chain $\Bpi$-complex.
Denote the chain complex of the first row  and the third row by respectively $\cK$ and $\cQ$.  The operators $\partial_0$ on the diagram indicate how to index the components of the complexes. For example, $\cQ_1=\cQ_2=\Bpi^n$.

The sequence
$$ 0\to\cK \to \sC \to \cQ \to 0$$
is split exact in each degree. Hence, for every $\Gamma\in \cG$, one has the following exact sequence of $\BZ$-complexes
\be \lbl{eq.ses1}
 0\to  \cK_\Gamma \to \sC_\Gamma \to \cQ_\Gamma \to 0.
\ee
Note that $\sC_\Gamma$ is the chain $\BZ$-complex of $Y_\Gamma$, and its homology groups are
the homology groups $H_*(Y_\Gamma,\BZ) = H_*(X_\Gamma,\BZ)$. We provide $\BZ[\GbP]$ with the metric in which $\GbP$
is an orthonormal basis. It is clear that this metric is integral, see Section \ref{sec:metricgroup}.

The short exact sequence \eqref{eq.ses1} generates a long exact sequence, part of it is
$$ H_2(\cQ_\Gamma) \overset {\beta_\Gamma} \longrightarrow  H_1(\cK_\Gamma) \to H_1(\sC_\Gamma) \to H_1(\cQ_\Gamma).$$
Note that $H_2(\cQ_\Gamma)  \le (\cQ_\Gamma)_2 = \BZ[\GbP]^n$ inherits an integral
 metric from $\BZ[\GbP]^n$. Similarly, $H_1(\cK_\Gamma) \le (\cK_\Gamma)_1 = \BZ[\GbP]$ inherits a integral metric.
Applying Lemma \ref{r.est}(b) to the above exact sequence, we get
 \be \lbl{eq.55}
 t(H_1(\sC_\Gamma)) \le \, {\det}'(\beta_\Gamma) \,  t (H_1(\cQ_\Gamma)) \, \vol(H_2(\cQ_\Gamma)).
   \ee
 Since   $H_1(\cQ_\Gamma) = \coker(J_\Gamma)$ and $H_2(\cQ_\Gamma) = \ker (J_\Gamma)$, by Lemma \ref{r.est}(a), we have
 $$  t (H_1(\cQ_\Gamma)) \, \vol(H_2(\cQ_\Gamma))  =  t (\coker(J_\Gamma)) \, \vol(\ker (J_\Gamma))  \le \det'(J_\Gamma).$$
 Using the above inequality in \eqref{eq.55}, and $t_1(\Gamma)=t(H_1(\sC_\Gamma))$, we get
\be\lbl{eq.56}
t_1(\Gamma) \le {\det}'(\beta_\Gamma) \, {\det}'(J_\Gamma).
\ee

{\it Claim.} $\det'(\beta_\Gamma)$ is a negligible function on $\cG$.
\begin{proof}
[Proof of Claim] We estimate $\det'(\beta_\Gamma)$  by using upper bounds for the rank and the norm of $\beta_\Gamma$. First, the rank of $\beta_\Gamma$ is less than or equal to rank of its codomain, which is $b_1(\cK_\Gamma)$.

The connecting homomorphism $\beta$ is the restriction of $\tilde \beta: \Bpi^n \to \Bpi$
given by $\tilde \beta(x) = x \cdot c$, where $c$ is the last column of $\partial_2$.
 It follows that $\|\beta_\Gamma \| \le \| \tilde \beta_\Gamma\| < \nu$ for some
 constant $\nu$ not depending on $\Gamma$ (see Lemma \ref{r.bound}). By \eqref{eq.norm},
$$ {\det}'(\beta_\Gamma) \le \nu^{b_1(\cK_\Gamma)}.$$
By Lemma \ref{r.circle}, $\lim b_1(\cK_\Gamma)/[\Pi:\Gamma]=0$. It follows that $\det'(\beta_\Gamma)$ is negligible. This completes the proof of the claim.
\end{proof}

 From \eqref{eq.56} and the above claim, we have
$$ t_1(\Gamma) \tle {\det}' J_\Gamma.$$
Hence
$$ \limsupG \frac{\ln t_1(\Gamma)}{ [\Pi:\Gamma]}  \le \limsupG \frac{\ln {\det}' (J_\Gamma)}{ [\Pi:\Gamma]} \le \ln  {\det}_\Pi J = \frac{\vol(X)}{6 \pi}, $$
where we use Theorem \ref{r.detbound} in the second inequality and Theorem \ref{thm.3} in the last equality.
\end{proof}
\def\cD{\mathcal D}
\def\cE{\mathcal E}

\subsection{Proof of Theorem \ref{thm.2a}} Let $X= S^3\setminus N(K)$, where $N(K)$ is a small  tubular open neighborhood of $K$. Then $\partial X$ is a torus. We will use
the notations of Section \ref{sec:none}.
 We can further assume that $a_{n+1}$ represents a meridian of $K$.

Suppose $\Gamma \in \cG$. Let $p:X_\Gamma \to X$ be the covering map. Then $p^{-1}(\partial X)$ consists of tori, and $p^{-1}(a_{n+1})$ is a collection of simple closed curves $C_1, \dots C_l$ on
$p^{-1}(\partial X)$. From the definition,  $H_1(\hat X_\Gamma,\BZ)= H_1(X_\Gamma,\BZ)/U$, where $U$ is the subgroup of $H_1(X_\Gamma,\BZ)$ generated by $C_1, \dots,C_l$.
The curves $C_1,\dots,C_l$ are made up from all the lifts of $a_{n+1}$.

The exact sequence \eqref{eq.ses1} gives rise to a long exact sequence
\be
\lbl{eq.ses2}
 \dots H_1(\cK_\Gamma) \overset{\gamma}{\longrightarrow} H_1(\sC_\Gamma) \to H_1(\cQ_\Gamma) \to H_0(\cK_\Gamma)\dots
 \ee
Recall that we use the identification $\sC_1 \equiv \Bpi^{n+1}$ via $\sC_1=\bigoplus_{j=1}^{n+1} \Bpi \cdot \tilde a_j$. Correspondingly, the identification
$\cK_1 \equiv \Bpi$ is via  $\cK_1 = \Bpi \cdot \tilde a_{n+1}$.
Under these identifications, one has $U\equiv \im (\gamma)$. Hence, from the exact sequence \eqref{eq.ses2}, one has the following exact sequence
$$ 0\to H_1(\hat X_\Gamma,\BZ)  \to H_1(\cQ_\Gamma) \to H_0(\cK_\Gamma)\dots$$
Since $H_0(\cK_\Gamma)$ is a free abelian group, the above exact sequence implies that
$$ \tor_\BZ (H_1(\hat X_\Gamma,\BZ))=  \tor_\BZ( H_1(\cQ_\Gamma)).$$
Applying \eqref{eq.detvol} to the map $J_\Gamma: \BZ[\GbP]^n \to \BZ[\GbP]^n$, we get
$$ |\tor_\BZ (H_1(\hat X_\Gamma,\BZ))|=  |\tor_\BZ( H_1(\cQ_\Gamma))| =   \frac{\det'(J_\Gamma)}{\vol(\ker(J_\Gamma)) \vol(\overline{\im(J_\Gamma))}} \le \det'(J_\Gamma).$$
Here $\overline{\im(J_\Gamma))}:= (\im(J_\Gamma) \ot_\BZ \BQ) \cap  \BZ[\GbP]^n$.

Theorem \ref{r.detbound} and Theorem \ref{thm.3} show
$$ \limsup_{\cG \ni \Gamma \totr 1} \frac{\ln |\tor_\BZ (H_1(\hat X_\Gamma,\BZ))|}{[\Pi:\Gamma]} \le
\limsup_{\cG \ni \Gamma \totr 1} \frac{\ln \det'(J_\Gamma)}{[\Pi:\Gamma]} \le \ln \det_{\Pi}(J)= \frac{\vol(X)}{6 \pi}.$$
This completes the proof of Theorem \ref{thm.2a}.

  \def\ind{\mathrm{ind}}

 \def\G{{\Gamma}}
 \def\ptD{ \partial_{2,\G}}
\def\cA{{\mathcal A}}
\def\cB{{\mathcal B}}
\def\cAG{ \tilde \cA}
\def\cBG{  \tilde \cB}
\def\tb{\tilde b}
 \def\tH{\tilde H}
 \def\tF{\tilde F}
  \def\cC{\mathcal C}

  \subsection{Heegaard splitting and homology} \lbl{sec.Hee1}
Assume that $X$ is an oriented connected closed 3-manifold and $X= H \cup H'$ is a Heegaard splitting of $X$. This means
each of $H$ and $H'$ is a   handlebody  of genus $g$, and $H \cap H'=\partial H =\partial H'$. For generalities on Heegaard splittings  the reader can consult \cite{Hempelbook}.

A properly embedded
disk $D \subset H$ is {\em essential} if its boundary does not bound a disk in $F=\partial H$. 
A {\em system of disks} of a handlebody $H$ of genus $g$ is a collection $\cD=\{D_1,\dots, D_n\}$ of properly embedded, essential oriented disks in $H$
 which are disjoint and cut $H$ into balls.
 Then $n \ge g$, and  if $n=g$, we say that the system of disks $\cD$ is {\em minimal}. A disk system $\cD$ is minimal if and only if  $H \setminus (\bigcup_{D \in \cD} D)$ is connected, and any disk system has a subset which is a minimal disk system.

Suppose $\cD=\{D_1,\dots,D_n\}$ is a disk system
 of $H$ and $\cD'=\{ D'_1, \dots, D'_m\}$ is a disk system of 
$H'$. Let $\al_i=\partial D_i$ and $\al_j'=\partial D'_j$.  The common boundary
 $F=\partial H=\partial H'$ inherits an orientation from $H$, and each curve $\al_i,\al'_j$ inherits an  orientation from $D_i, D'_j$. We will assume that $\al_i$ and $\al'_j$ are transversal  in $F$ for all pairs $(i,j)$, and let $\mu(\al_i,\al'_j)$ be the intersection index of $\al_i$ and $\al'_j$.

The disk  system $\cD$ generates a {\em dual graph} $G$ embedded in $H$: to every connected component $Q$ of $H \setminus (\cup_{i=1}^n D_i)$ there corresponds a vertex $v_Q$ which is a point in the interior of $Q$, and to every disk $D_i$ there corresponds an edge $a_i\subset H$. If $D_i$ is in the closure of two connected components $Q$ and $Q'$, then $a_i$ is an edge connecting $v_Q$ and $v_{Q'}$. If $D_i$ is in the closure of only one connected component $Q$, then $a_i$ is a loop edge based at $v_Q$. In all cases, $a_i$ intersects $D_i$ transversely at exactly one point, and $a_i \cap D_j =\emptyset$ for $i\neq j$.  We orient $a_i$ so that the intersection index $\mu(a_i, D_i)=1$.  Note that $G$, known as a spine of $H$, is a deformation retract of $H$. Let $G'\subset H'$ be a dual graph of $\cD'$, with edge $a'_i$ dual to disk $D'_i$.

 Although $a_i$ and $a'_j$ do not intersect, we will define
 \be
 \mu(a'_i,a_j)= \mu(\al'_i, \al_j).
 \ee

  \begin{proposition} \lbl{r.tor1} Let $M\in \Mat(m\times n, \BZ)$ be given by $M_{ij}= \mu(a_i',a_j)$.  Then $t(H_1(X,\BZ)) = t(M)$.\end{proposition}
  \begin{proof}

 Without loss of generality we can assume that $\bar \cD=\{D_1,\dots, D_g\}$ is a minimal disk system of $H$. Let $\bar M$ be the $m\times g$ submatrix of $M$ consisting of the first $g$ columns.
 The dual graph $\bar G$ of $\bar \cD$ has only one vertex and  $g$ loop edges $\bar a_1, \dots, \bar a_g$ (dual to $D_1,\dots, D_g$). The group $H_1(H,\BZ)$ is free abelian with basis $\{\bar a_1, \dots, \bar a_g\}$. If $c\subset F=\partial H$ is an oriented closed curve, then the homology class of $c$ in $H_1(H,\BZ)$ is $\sum_{i=1}^g \mu(c, \al_i) \bar a_i$. 
  To obtain $X$ one glues $H'$ to $H$, and this can be done in 2 steps. In the first one glues the disks $D'_j$ to $H$, then in the second one glues in the complements (in $H'$) of $D'_i$ which are 3-balls. The second step does not change the  homology group $H_1$. The first gluing shows that 
  $H_1(X,\BZ)= \coker  \CR_{\bar M}$, where $\CR_{\bar M} (x)= x\bar M$.

Cutting $F=\partial H$ along  $\al_1,\dots,\al_g$, we get a sphere with $2g$ disks removed. More precisely,  $F$ is obtained from $F'$, a compact surface of genus 0 and $2g$ boundary components $c'_1,c''_1, c'_2,c''_2,\dots, c'_g, c''_g$, by gluing $c'_i$ to $c''_i$, with  the common image being $\al_i$.
Any simple closed curve in the interior of $F'$ is separating (because $F'$ has genus $0$), and hence is homologically equal to a $\BZ$-linear combination of boundary curves  with coefficients $\pm1$.
It follows that the homology class of any curve in $F$ not meeting any of $\al_1,\dots,\al_g$ is in the $\BZ$-linear span of $\al_1,\dots, \al_g$; this applies to the curves $\al_i$ with $i>g$.
This implies any column of $M$ is a $\BZ$-linear combination of the first $g$ columns. In other words, the images of $\CR_{M^*}$ and $\CR_{\bar M^*}$ are same. Hence 
$ t(M^*)= t(\bar M^*)$. By Proposition \ref{r.tor5}, one has $t(M)= t(M^*)$. It follows that  $t(H_1(X,\BZ)) = t(\bar M)=t(M)$.
  \end{proof}
  
 Here is another proof of Proposition \ref{r.tor1}. A Heegaard splitting $X= H\cup H'$ and disk systems $\cD$ of $H$ and $\cD'$ of $H'$ give rise to a CW-complex structure of $X$ as follows. First, the handlebody $H$ is obtained from zero-handles, each is a regular neighborhood of a vertex of the graph $G$ dual to $\cD$, by attaching one-handles whose cores are in the edges $a_i$ of $G$. Then to $H$ one glues two-handles whose cores are $D'\in \cD$. Finally by gluing in three-handles one gets  $X$.  This handle decomposition gives rise to a CW-complex structure of $X$, see \cite[Chapter 6]{RS}. The second boundary map of the associated chain complex is given by the matrix $M$.
 By Proposition~\ref{r.tor5}(c), we get $t(H_1(X,\BZ)) = t(M)$.

 Recall that a cycle of a graph is a closed walk in the graph which does not visit any vertex twice. We show here a way to simplify a disk system of a handlebody.

\begin{lemma} \lbl{r.exhaust}
 Suppose $c_1, \dots, c_r$ are disjoint cycles of $G$. In each cycle choose an edge, called the preferred edge of the cycle. Then the   set $\cE\subset \cD$, consisting of all $D_i$ such that either $a_i$ is a preferred edge of a cycle or $a_i$ does not belong to any cycle,  is a system of disks of $H$.
\end{lemma}

\begin{proof}  A cycle of $G$ either (i)  is  one of $c_1, \dots, c_r$, or
(ii) contains an edge which is not in any of $c_1, \dots, c_r$. This shows $\cE$ cuts $G$ into trees, implying each connected component of $H \setminus (\cup _{D \in \cE}D)$  is contractible.  Hence $\cE$ is a disk system  of $H$.
\end{proof}

\def\tf{\tilde f}
\def\tG{\tilde G}
 \def\AG{{\mathcal A}_\Gamma}
  \def\BG{\tilde{\mathcal B}}
    \def\tA{\tilde {\cA}}
    \def\tB{\tilde{\cB}}
    \def\cU{\mathcal U}
    \def\tU{\tilde{\cU}}
    \def\tD{\tilde {\cD}}
   \def\tJ{\tilde J}
   \def\tS{\tilde S}
 \newcommand{\red}[1]{{\color{red}#1}}

  \subsection{Proof of Theorem \ref{thm.2} for the case when $\partial X=\emptyset$}
  Suppose $X$ is an irreducible, connected, closed, oriented 3-manifolds. If $\Pi=\pi_1(X)$ is finite, then the statement of Theorem \ref{thm.2} is trivial. We will assume that $\Pi$ is infinite. Then every non-trivial element of $\Pi$ has infinite order.

Suppose $X= H \cup H'$ is a Heegaard splitting of $X$, where $F=\partial H=\partial H'$ has genus $g$.  Suppose $\cD=\{D_1,\dots,D_g\}$ is a minimal disk system
of $H$ and $\cD'=\{ D'_1, \dots, D'_g\}$ is a minimal disk system of 
$H'$. Let the dual graph $ G$ (resp. $G'$) with the set of edges $\cA=\{a_1, \dots, a_g\}$ (resp. $\cA'=\{a'_1, \dots, a'_g\}$) be as defined in Subsection \ref{sec.Hee1}. The minimality of $\cD$ and $\cD'$ implies each of  $G$ and $G'$ has  one vertex, and each of $a_i$, $a_i'$ is a loop-edge. Each of the sets $\{a_1,\dots, a_g\}$ and $\{a'_1,\dots, a'_g\}$ generates $\Pi=\pi_1(X)$. Hence by reordering, one can assume that $a_g$ and $a'_g$ are non-trivial in $\Pi$. Thus, the associated presentation of $\Pi$,
   \be
  \lbl{eq.presen2}
   \Pi= \la a_1,\dots, a_{g} \mid r_1,\dots, r_{g}\ra,
   \ee
 where $r_i$ is determined by the two-cell  $D'_i$, is good.

 {Let $\tJ \in \Mat(g \times g, \BZ[\Pi])$ be the matrix whose $ij$-entry is given by
  $ \tJ_{ij} = \frac{\partial r_i}{\partial a_j}.$ 
  Then $\tJ$ is the second boundary operator of the chain complex of the universal covering of $X$, see Subsection \ref{sec:chainC}.  The reduced Jacobian $J$ is obtained from $\tJ$ by removing the last row and the last column.

 Let $(\G_m)_{m\ge 0}$ be a sequence of subgroups of $\Pi$ of finite index such that $\G_m \totr 1$.
  Fix an index $m$ for now.
 Let $P_m: X_m \to X$ be the covering map corresponding to the subgroup $\G_m$ and $\tH, \tH', \tG, \tG'$ be respectively the preimage of $H,H',G,G'$ under $P_m$.

For a disk $D$ in $\cD$ or $\cD'$, the preimage $(P_m)^{-1}(D)$ is a collection of disjoint disks, each is called a {\em lift} of $D$.
 The collection $\tD$ (resp. $\tD'$) of all lifts of all disks in $\cD$ (resp. $\cD'$) is a disk system of $\tH$ (resp. $\tH'$). 
  The CW-structure of $X_m$, corresponding to the Heegaard splitting $X_m = \tH \cup \tH'$ and the disk systems $\tD$ and $\tD'$, is exactly the lift of the CW-structure of $X$.
Hence the boundary operator $\partial_2: \cC_2(X_m) \to \cC_1(X_m)$ is equal to $ \tJ_m:=\tJ_{\G_m}$.

Let $\tA$ (resp. $\tA'$) be the set of edges of $\tG$ (resp. $\tG'$). An edge $\tilde a$ of $\tilde G$ (resp. $\tilde G'$) is called a {\em lift} of an edge $a$ of $G$ (resp. $G')$ if $P_m(\tilde a)= a$.
As explained in Subsection \ref{sec.Hee1}, if we identify $\cC_2(X_m)$ with the free $\BZ$-module with basis $\tA'$ and identify $\cC_1(X_m)$ with the free $\BZ$-module with basis $\tA$, then  $\tJ_m=\partial_2: \cC_2(X_m) \to \cC_1(X_m)$ is  given by the $\tA'\times \tA$ matrix  whose $(a',a)$-entry is  $\mu(a', a)$.

Suppose all the lifts of $a_{g}$ form $l_m$ cycles. 
In each cycle choose one lift of $a_{g}$, called a {\em preferred lift}.
 Similarly, all the lifts of $a'_{g}$ form $l'_m$ cycles. In each cycle choose one lift of $a'_{g}$, called a {\em preferred lift} of $a'_{g}$.
Let $\BG$ (resp. $\BG'$)  be obtained from $\tA$ (resp. $\tA'$) by removing all the non-preferred lifts of $a_g$ (resp. $a'_g$).
Let $J'_m$ be the $\BG' \times \BG$-submatrix of $\tJ_m$.
By Lemma \ref{r.exhaust} and Proposition \ref{r.tor1},  one has $ t(H_1(X_m))= t(J'_m)$, hence by Proposition \ref{r.tor5}},
$$ t(H_1(X_m))= t(J'_m) \le {\det}' (J'_m).$$

 By Theorem \ref{r.detbound}, the sequence $(J_m, [\Pi:\G_m])_{m \ge 1}$ tracely approximates $J$, where
 $J_m:= J_{\G_m}$. Note that $J_m$  is obtained from $J'_m$ by removing $l_m$ rows corresponding to the $l_m$ preferred lifts of $a_{g}$ and $l'_m$ columns corresponding to the $l'_m$ preferred lifts of $a'_{g}$. Let $d_m= \max(l_m, l'_m)$. By (the proof of) Lemma~\ref{r.circle},  $\lim_{m\to \infty}\frac{d_m}{[\Pi:\G_m]}=0$. Besides, by Lemma \ref{r.bound}, there is $L>0$ such that $\|\tJ_{\G_m}\| <L$. As a submatrix of $\tJ_{\G_m}$, we also have $\|J_m'\| < \|\tJ_{\G_m}\|<L$. Now Proposition \ref{r.detbound3} (see Remark \ref{rem.upp}) shows the second inequality in the following

 $$ \limsup_{m\to \infty} \frac{\ln t(H_1(X_m))}{ [\Pi:\Gamma_m]}
  \le \limsup_{m\to \infty}  \frac{\ln {{\det}}' (J'_m)}{ [\Pi:\Gamma_m]} \le \ln {{\det}}_\Pi J = \frac{\vol(X)}{6 \pi}, $$
which proves Theorem \ref{thm.2}.

 \begin{remark}
 One can have a proof of Theorem \ref{thm.2} for the case when $\partial X \neq \emptyset$ similar to (and actually simpler than)  the above proof for the case when   $\partial X= \emptyset$.
 \end{remark}
 \appendix

\section{Proof of Theorem \ref{r.detbound1}}  \lbl{sec.a1}
 \def\uF{\underline G}
 \def\oF{\overline G}
 \def\nc{\mathrm{nc}}
  Assume $(B_m,N_m)_{m\ge 1}$ tracely approximates $B$, and $K$ is the constant in Definition \ref{def.appro}.
  Let $F$ be the spectral density function of $B$, and $F_m$ be the spectral density function of $B_m$.  Define
 \be 
 \lbl{eq.def5}
  G_m (\lambda)= \frac{F_m(\lambda)}{N_m}, \quad \uF(\lambda) = \liminf_{m\to \infty} G_m (\lambda), \quad \oF(\lambda) = \limsup_{m\to \infty} G_m (\lambda).
  \ee
 Recall that for any  increasing function $h:\BR \to \BR$
 one defines $   h^+(\lambda)= \lim_{\ve \to 0^+} h(\lambda + \ve).$

\begin{lemma}  \lbl{r.tech}
One has $F= \uF^+=\oF^+$. \end{lemma}

\begin{proof} The proof is a slight modification of that of \cite[Theorem 2.3]{Luck:approximation}.

Since $\|B\| ,  \|B_m B^*_m \| < K^2  $, we have $F_m(\lambda)= \nr(B_m)$ and $F(\lambda)= \nr(B)$ for all  $\lambda\in[{K^2},\infty)$. Applying Condition (ii) of Definition~\ref{def.appro} to the constant polynomial $p=1$ and $\lambda =K^2$, one has

\be
 \lbl{eq.17a}
 \lim_{m\to \infty} \frac{\nr(B_m)}{N_m}= \nr(B).
\ee

Fix a number $\lambda\in [0,{K^2})$ and let $(p_n)_{n\ge 1}$ be the sequence of polynomials constructed in \cite[Section 2]{Luck:approximation}. These polynomials approximate the  characteristic function of the interval $[0,\lambda]$ and have the following properties:

\be \lbl{eq.appr1}
0\le p_n  \ \text{on} \  [0,{K^2}], \quad  1\le p_n \le 1+ \frac 1n \ \text{on} \  [0,\lambda],
 \ee \be
 \lbl{eq.appr2}  p_n \le 1+ \frac 1n \ \text{on} \  [\lambda, \lambda + \frac 1n],
 \quad    p_n(\mu)  \le \frac 1n \ \text{on } \ [\lambda + \frac 1n, {K^2}].
\ee \be
\lbl{eq.appr4}
\lim_{n\to \infty} \tr_\Pi(p_n(BB^*)) = F(\lambda).
\ee

 Using the bounds of $p_n$ given by \eqref{eq.appr1}, we have
 $$\tr(p_n(B^*_m B_m))  =  \int_0^{K^2}  p_n(\mu) dF_m(\mu)  \ge  \int_0^\lambda  p_n(\mu) dF_m(\mu)
 \ge  \int_0^\lambda   dF_m(\mu) = F_m(\lambda).
 $$
  Dividing by $N_m$,
 $$ G_m(\lambda)=\frac{F_m(\lambda)}{N_m}  \le \frac{\tr(p_n(B^*_m B_m))}{N_m}.
 $$
Take the superior limit as $m\to\infty$, then the limit as $n\to \infty$. Using Condition (ii) of Definition~\ref{def.appro} and \eqref{eq.appr4}, we have
 \be \lbl{eq.o}
  \oF(\lambda) \le F(\lambda).
 \ee

 Use the bounds of $p_n$  given by \eqref{eq.appr2},  we have
   \begin{align*}
\tr(p_n(B^*_m B_m)) & =  \int_0^{K^2}  p_n(\mu) dF_m(\mu)  =  \int_0^{\lambda + 1/n}  p_n(\mu) dF_m(\mu)  + \int_{(\lambda + 1/n)^+}^{K^2}  p_n(\mu) dF_m(\mu)
 \\
&\le \left( 1 + \frac 1n \right) F_m\left(\lambda + \frac 1n\right) + \frac 1 n \left (F_m({K^2}) - F_m\left(\lambda + \frac 1n\right)\right)\\
&=   F_m\left(\lambda + \frac 1n\right) +  \frac 1n F_m({K^2}) =  F_m\left(\lambda + \frac 1n\right) +   \frac {\nr(B_m)}n
 \end{align*}
Dividing by $N_m$, then taking the inferior limit as $m\to \infty$ and using \eqref{eq.17a}, we have
$$ \tr(p_n(B^* B) \le \uF(\lambda + \frac 1n) + \frac{\nr(B)}{n}.$$
Taking  the limit as $n \to \infty$, we have
\be
\lbl{eq.l}
F(\lambda) \le \uF^+(\lambda).
\ee
Using \eqref{eq.l}, \eqref{eq.o} and the fact that $\uF$ is increasing, for all $\ve >0$ we have
 $$
F(\lambda) \le \uF^+(\lambda) \le \uF(\lambda + \ve) \le \oF(\lambda + \ve) \le F(\lambda + \ve).
$$
Taking the limit $\ve \to 0^+$ and noting that $F$ is right continuous, i.e. $\lim_{\ve \to 0+} F(\lambda + \ve ) = F(\lambda)$, we have $F(\lambda) = \uF^+(\lambda)= \oF^+(\lambda)$.
\end{proof}
\begin{lemma}\lbl{r.a2}
 One has
 \begin{align}
 \lbl{eq.9i}
 \uF^+(0)
 &= \lim_{m \to \infty} G_m(0),\\
 \lbl{eq.9j}  \uF^+({K^2})& = \lim_{m \to \infty} G_m({K^2}). \end{align}
\end{lemma}
\begin{proof} We have \eqref{eq.9i}, since
  the left hand side is equal to $F(0)$ by Lemma \ref{r.tech}, while the right hand side is also equal to $F(0)$ by Condition  (iii) of Definition \ref{def.appro}.

  Since each $G_m$, like $F_m$, is a constant function on $[{K^2},\infty)$, each of $\uF$ and $\oF$ is constant on $[{K^2},\infty)$. It follows that $\uF= \uF^+$ and $\oF= \oF^+$ on $[{K^2},\infty)$. Hence for $\lambda \in [{K^2},\infty)$,
  $$ \uF(\lambda)= \uF^+(\lambda) = F(\lambda) =  \oF^+(\lambda)= \oF(\lambda).$$
  This implies $\lim_{m\to \infty} G_m(\lambda)= \uF(\lambda)= \uF^+(\lambda)$ for $\lambda \in [{K^2},\infty)$. In particular,
  we have \eqref{eq.9j}.\end{proof}

  Let us prove Theorem \ref{r.detbound1}.
By definition and Lemma \ref{r.tech},  one has $({\det}_\Pi(B))^2=\det(F)= \det(\uF^+)$.  Using
 \eqref{eq.detpart} then Lemma \ref{r.a2}, we have
\begin{align}
 2 \ln ({\det}_\Pi(B)) = \ln \det(\uF^+)&= (\uF^+({K^2})-\uF^+(0)) \ln({K^2}) -  \int_{0^+}^{K^2} \frac{\uF^+(\lambda)- \uF^+(0)}{\lambda} d \lambda  \nonumber
\\ &= \lim_{m\to \infty} (G_m({K^2}) - G_m(0)) \ln {K^2} - \int_{0^+}^{K^2} \frac{\uF^+(\lambda)- \uF^+(0)}{\lambda} d \lambda
\lbl{eq.011}
\end{align}

By  \cite[Lemma 3.2]{Luck:approximation},  we the first of the following identities
\begin{align}   \int_{0^+}^{K^2} \frac{\uF^+(\lambda)- \uF^+(0)}{\lambda} d \lambda &=
  \int_{0^+}^{K^2} \frac{\uF(\lambda)- \uF^+(0)}{\lambda} d \lambda
  \nonumber
   \\
   &= \int_{0^+}^{K^2} \liminf_{m \to \infty }\frac{G_m(\lambda)- G_m(0)}{\lambda} d \lambda \notag\\
 &\le \liminf_{m \to \infty }  \int_{0^+}^{K^2} \frac{G_m(\lambda)- G_m(0)}{\lambda} d \lambda, \lbl{eq.ine4}
 \end{align}
where the second identity follows from $\uF(\lambda)= \liminf_{m\to \infty} G_m(\lambda) $ and \eqref{eq.9i}, and the
  last inequality follows  from Fatou's lemma.
Using Inequality \eqref{eq.ine4} in \eqref{eq.011}, we have
  \begin{align*}
  2 \ln ({\det}_\Pi(B)) & \ge \lim_{m\to \infty} (G_m({K^2}) - G_m(0)) \ln {K^2} - \liminf_{m\to \infty} \int_{0^+}^{K^2} \frac{G_m(\lambda)- G_m(0)}{\lambda} d \lambda \\
 &= \limsup_{m\to \infty}\left\{  (G_m({K^2}) - G_m(0)) \ln {K^2} - \int_{0^+}^{K^2} \frac{G_m(\lambda)- G_m(0)}{\lambda} d \lambda \right \}\\
 & =  \limsup_{m\to \infty} \{ 2 \ln ({\det}' (B_m))/N_m\},
  \end{align*}
where in the last equality we again use \eqref{eq.detpart}. This completes the proof of Theorem \ref{r.detbound1}.

\begin{remark} As observed by the referee, Theorem \ref{r.detbound1} follows from the portmanteau theorem in probability theory as follows (sketch). Let $G_m,F$ be defined as in \eqref{eq.def5}. Then $P_m:=dG_m$ and $P:= dF$ are finite Borel measures  on $\BR$, with support in $[0,K^2]$. Condition (ii) of Definition \ref{def.appro} implies that for any polynomial $f$,
\be 
\lbl{eq.0f} \lim_{m\to \infty} \int_{[0,K^2]} f d  P_m = \int_{[0,K^2]} f d P
\ee
Every continuous function on $[0,K^2]$ can be uniformly approximated by polynomials. It follows that \eqref{eq.0f} holds for all bounded continuous functions. Hence, by  definition (see \cite{Ambrosio}), the sequence of Borel measures $P_m$ {\em converges weakly} (or {\em converges narrowly}) to $P$. 

Let $P'_m,P'$ be respectively the restriction of the measures $P_m,P$ on $(0,K^2]$. Condition (iii) of Definition \ref{def.appro} means that $\lim_{m\to \infty} P_m(0)= P(0)$. From here and the weak convergence of $P_m$ to $P$, one can prove that the sequence $P'_m$ weakly converges to $P'$. By a version of the pormanteau theorem (see \cite[Section 5.1.1]{Ambrosio}, in particular Formula (5.1.15) there), one has  that if $h$ is a function on $(0,K^2]$ which is continuous and bounded from above, then
$$ \limsup_{m\to \infty} \int_{(0,K^2]} h d P'_m \le \int_{(0,K^2]} h dP'.$$
Applying to $h=\ln(x)$, we get 
$$ \limsup_{m\to \infty} (\ln\det G_m) \le \ln\det F,$$
which proves Theorem \ref{r.detbound1}.
\end{remark}

\end{document}